\documentclass{birkjour_t2}

\usepackage{amssymb}

\usepackage{graphicx}

\usepackage{comment}
\usepackage{mathrsfs}
\usepackage[colorlinks=true]{hyperref}
\hypersetup{urlcolor=blue, citecolor=red, linkcolor=blue}
\usepackage{calc}
\usepackage{color}
\usepackage{marginnote}
{\begin{list}{\arabic{enumi}.}{\usecounter{enumi}%
    \setlength{\labelsep}{0.5em}%
    \settowidth{\labelwidth}{(\arabic{enumi})}%
    \setlength{\leftmargin}{\labelwidth+\labelsep}}}%
{\end{list}}

\usepackage[symbol]{footmisc}
\newcommand{\R}        	{\mathbb{R}}
\newcommand{\mL}    	{\mathscr{L}}
\newcommand{\D}        	{\Delta}

\newcommand{\p}			{\partial}

\newcommand{\Oo}		{\mathcal{O}}
\newcommand{\T}			{\mathbf q}
\newcommand{\col}[1]{\begin{bmatrix}#1\end{bmatrix}}
\newcommand{\bra}[1]{\left[#1\right]} 
\newcommand{\mat}[1]{\begin{matrix}#1\end{matrix}} 
\newcommand{\bmat}[1]{\bra{\mat{#1}}} 

\def \sig{\kappa}
\def \q{q}
\def \Ups{ \R^n}
\def\R{\mathbb R}

\def \k{\underline{k}}
\def
\ord1{  \nabla_x(\rho u \cdot \nabla_x \Phi) -
(\nabla_x \cdot \rho u) \nabla_x \Phi}

\def \dPhi{ \nabla_x \Phi}

\newtheorem{theorem}{Theorem}[section]
\newtheorem{Lemma}[theorem]{Lemma}
\newtheorem{Proposition}[theorem]{Proposition}

\theoremstyle{definition}
\theoremstyle{definition}
\newtheorem{Definition}[theorem]{Definition}
\newtheorem*{Assumption}{Assumption}

\theoremstyle{remark}
\newtheorem{Remark}[theorem]{Remark}

\numberwithin{equation}{section}
\begin{document}

\title[Inverse problem for inhomogeneous non-smooth waves]{Recovery of piecewise smooth parameters in an acoustic-gravitational system of equations from exterior Cauchy data }


\author{Sombuddha Bhattacharyya}
\address{Indian Institute of Science Education and Research Bhopal}
\email{sombuddha@iiserb.ac.in}
\thanks{}

\author{Maarten de Hoop}
\address{Rice University}
\thanks{}

\author{Vitaly Katsnelson}
\address{New York Tech}
\email{vkatsnel@nyit.edu}
\thanks{}



\date{January 1, 2004}
\dedicatory{}

\begin{abstract}
In this paper, we study an inverse problem for an acoustic-gravitational system whose principal symbol is identical to that of an acoustic wave operator. The displacement vector of a gas or liquid between the unperturbed and perturbed flow is denoted by $u(t,x)$. It satisfies a partial differential equation (PDE) system with a principal symbol corresponding to an acoustic wave operator, but with additional terms to account for a global gravitational field and self-gravitation. These factors make the operator nonlocal, as it depends on the wave speed and density of mass. We assume that all parameters are piecewise smooth in $\mathbb{R}^3$ (i.e., smooth everywhere except for jump discontinuities across closed hypersurfaces called interfaces) but unknown inside a bounded domain $D$. We are given the solution operator for this acoustic-gravitational system, but only outside $D$ and only for initial data supported outside $D$. Using high-frequency waves, we prove that the piecewise smooth wave speed and density are uniquely determined by this map under certain geometric conditions.

\end{abstract}

\maketitle

\section{Introduction}
There is a substantial body of literature on inverse problems aimed at recovering multiple coefficients of a wave operator simultaneously, including coefficients that are not present in the principal symbol of the wave operator (see, for example, \cite{OksanenLowerOrder, BHKUdensityInterior, Montalto2014}). However, these prior works are incomplete in many real-world scenarios, such as the elasto-gravitational equations modeling seismic waves, where non-local wave interactions arise due to the planet's self-gravitation. Most previous studies, by contrast, deal with local wave operators.

These gravitational interactions can be modeled by a zeroth-order pseudodifferential operator, adding complexity to the analysis of inverse problems. This is because the principles of unique continuation and finite speed of propagation no longer apply. Additionally, measuring amplitudes using high-frequency waves or Gaussian beams at the boundary does not resolve the issue. The non-local component of the wave operator naturally manifests in the ``lower order" amplitudes of the waves, requiring direct disentanglement and analysis of these non-local wave contributions.

Our goal is to image the wave speed and density of mass in more comprehensive planetary models that include both the gravitational potential and self-gravitation, using only exterior measurements. In addition, the coefficients are smooth but are discontinuous across a set of of closed hypersurfaces, which creates multiple scattering in the wave data. This problem falls into the category of boundary inverse problems for an elastic system, derived from the dynamic Dirichlet-to-Neumann map. However, the most recent works \cite{Oksanen2020, BHKUdensityInterior, RachDensity}, which uniquely determine wave speed and density from boundary measurements, remain incomplete because they use a local wave model that overlooks the effects of self-gravitation and the gravitational potential. These effects are inherently dependent on mass density in a nontrivial way.

In this work, we address this gap by solving an inverse problem for an acousto-gravitational system that incorporates self-gravitation and a gravitational field, while at the principal symbol level, is an acoustic wave equation. We demonstrate how to disentangle the non-local wave interaction components in the data to image both wave speed and density. Although we study an equation whose principal symbol is scalar to simplify computations, our analysis extends to the full elasto-gravitational system as well (see section \ref{s: discussion}).

Mathematically, we will uniquely recover the wave speed and other coefficients of a second-order acousto-gravitational initial-boundary value problem (see \eqref{IVP}) from measurements taken outside a bounded domain 
$\Omega \subset \R^3$. This system involves both a hyperbolic partial differential operator and lower-order terms that model the reference gravitational force and the planet's self-gravitation. The corresponding operator is non-local and depends globally on the density of mass. The study aims to understand and quantify the impact of this non-local term on the inverse problem.
Additionally, the coefficients are assumed to be piecewise smooth in the closure of the domain, with potential jumps at interfaces separating different subdomains. To facilitate the analysis, the coefficients are extended to the entire three-dimensional space outside the domain and are assumed to be known outside 
$\Omega$. While the hyperbolic part governs the propagation of singularities, finite speed of propagation no longer applies due to the pseudodifferential component. Although we study an equation whose principal symbol is scalar to simplify computations, our analysis extends to the full elasto-gravitational system as well (see section \ref{s: discussion}).

Even in the local elastic wave setting, it was proven in \cite{RachDensity} that density does not appear in the ``leading order amplitude" of a wave solution, which can also be interpreted as the principal symbol of the Dirichlet-to-Neumann map. Instead, it only begins to appear in the ``lower order polarization," demonstrating that the leading order amplitudes of elastic waves contain no information about the density. Our proof shows that both the gravitational potential and self-gravitation appear in the second to leading-order wave amplitudes that contain information on the density. Therefore, such gravitational terms cannot be simply ignored in the analysis due to being ``lower order" in the operator; in other words, although these terms do not appear in the principal symbol of the elastic wave operator, these terms cannot be ignored when recovering the density even when using microlocal techniques. We demonstrate how to disentangle the density from the gravitational terms in order to image density.

 Another novelty is that our proof shows that measuring wave amplitudes in the exterior also does not suffice to determine the density if we have no \emph{a priori} information on the reference gravitational potential $\Phi$. Hence, we must assume that $\Phi$ is known at the surface of the object. It is important to note that knowledge of $\Phi$ at the surface by itself also does not give us significant information about $\rho$ in the interior. In \cite{IsakovGravimetry}, the authors show that knowledge of $\Phi$ at the surface does not even determine a characteristic function $\chi_\Omega$, where $\Omega$ is some domain, unless it falls within a certain class. Hence, we necessarily need elastic wave amplitudes as well. Thus, this inverse problem takes on a coupled form where both a hyperbolic equation and an elliptic equation are analyzed simultaneously for imaging a coefficient.

 \begin{Remark}
 Knowledge of $\Phi$ at the surface is a reasonable assumption in practice (see \cite{InverseGrav}) and is related to a classical inverse problem of gravimetry. In the inverse gravimetry problem, $\nabla_x \Phi$ is assumed to be measured at the surface of the object rather than $\Phi$, and one tries to obtain information about $\rho$ in the interior \cite{IsakovGravimetry}. However, in our setting, both the wave operator and the transmission conditions only involve $\nabla_x \Phi$ and not $\Phi.$ Hence, our measurement data in theorem \ref{Main_th_2} is unchanged if we replace $\Phi$ by adding to it a piecewise constant function, so we could have used knowledge of $\nabla_x\Phi$ outside of $\Omega$ rather than $\Phi$ as well.
\end{Remark}

The displacement vector of a gas or liquid parcel between the unperturbed and perturbed flow is $u$. 
The 
 bounded domain is denoted $\Omega\subset \mathbb R^3$ and assume each coefficient is a piecewise smooth, positive function in the closure of the domain. We consider the wave operator on $\mathbb R_{+} \times \R^3$ as
\begin{equation}\label{Wave_op}
\mL = \mL_{\kappa,\rho} := \rho \p_{t}^2 - \nabla_x \kappa \text{div}_x + \ord1 + \rho \nabla S \rho, \qquad (t,x) \in \R_+\times\R^3,
\end{equation}
where $\kappa(x),\rho$ are real-valued functions, and $S$ is the operator defined by 
\[
\Delta S(\rho u) = -4\pi G \nabla \cdot \rho u
\]
with gravitational constant $G$. Also $\Phi$ satisfies
\[
\Delta_x \Phi = 4\pi G \rho 
\]
with standard transmission conditions across interfaces. We also have the \emph{hydrostatic pressure} $p^0$ that satisfies
\[
\nabla_x p^0 = - \rho \nabla_x \Phi.
\]

We assume that the coefficients $\rho, c$ of the PDE are ``piecewise smooth'' in $\overline{\Omega}$, which we now describe. 
Let $\{\Gamma_j\}_{j=0}^{m}$ be a collection of smooth, closed, hyper-surfaces that we term \emph{interfaces} such that $\Gamma_0 = \partial\Omega$ and $\Gamma:= \cup_{j=0}^{m} \Gamma_j$, which split $\Omega$ into a finite collection of subdomains. We assume that each coefficient we consider is smooth ($C^\infty$) in $\overline{\Omega}$ up to these interfaces, with possible jumps there. Thus, the PDE coefficients are in $C^\infty(\overline{\Omega}\setminus \Gamma)$; see \cite[Figure 1]{SUV2019transmission} for an illustration of such a domain. For mathematical convenience, to allow us to work outside of $\Omega$, we assume all coefficients have been extended to $\R^n$ and are in $C^\infty(\R^n \setminus \Gamma)$. 

It will be useful to denote 
\begin{equation}
B(x,D) = \nabla S
\end{equation}
which may be thought of as a pseudodifferential operator of order $0$ in $\Psi(\R^n)$.

\begin{Remark} \label{examp: elliptic B}
In fact, the equation $\mathcal{L}u = 0$ can be expressed as a matrix system of differential operators in the form
\[
\rho \p_t^2 u - \nabla_x \kappa \nabla_x \cdot u
+ \ord1 +\rho \nabla \Psi(x) = 0
\]
where $(\Psi, \Phi)$ satisfies the equation
\[
\Delta \Psi = -4\pi\nabla \cdot (\rho\ u), \qquad \Delta_x \Phi = 4\pi G \rho
\]
subject to certain boundary and interface conditions determined by a variational principle in the Euler-Lagrange equations. Hence, this may be thought of as a coupled physics problem with hyperbolic and elliptic constituents. Other works that contain similar equations, but in a systems of PDE setting may be found in \cite{Saturn,LyndenBell}. However, in the current work, we do not consider this perspective, as the analysis simplifies considerably when $\nabla S$ is thought of merely as a pseudodifferential operator. There have been a growing number of articles studying wave equations involving a zeroth order, self-adjoint pseudodifferential operator(see \cite{DyatlovForced,ColinForced} and the references there).
\end{Remark}

\subsection*{History of the problem}
In general, the hyperbolic inverse problem asks to determine the unknown coefficient(s), representing wave speeds, of a wave equation inside a domain of interest $\Omega$, given knowledge about the equation's solutions (typically on $\p \Omega$). Traditionally for inverse problems, the coefficients are taken to be smooth, and the data is the Dirichlet-to-Neumann (DN) map, or its inverse. The main questions are the uniqueness and stability of the coefficients: Can the coefficients be recovered uniquely from the Dirichlet-to-Neumann map, and is this recovery stable relative to the perturbations in the data? In the case of a scalar wave equation with smooth coefficients, several results by Belishev, Stefanov, Uhlmann, and Vasy \cite{Belishev-multidimenIP, UVLocalRay, SURigidity} have answered the question affirmatively. For the piece-wise smooth case, a novel scattering control method was developed in \cite{CHKUControl} to show in \cite{CHKUUniqueness} that uniqueness also holds for piece-wise smooth wave speeds with conormal singularities, under mild geometric conditions. In the elastic setting, in \cite{SUV2019transmission}, Stefanov, Uhlmann, and Vasy recover piece-wise smooth wave speeds from the Dirichlet-to-Neumann map under certain geometric conditions. That paper essentially recovers coefficients appearing in the principal symbol of the partial differential operator but does not address recovering piece-wise smooth lower-order coefficients such as the density of mass.
In the smooth elastic case, Rachele in \cite{RachDensity} recovers the density for the isotropic elastic wave equation for smooth coefficients.  
 By combining the general results of Weinstein in \cite{weinstein1976} together with ``scattering control'' described in \cite{CHKUElastic}, it was shown in \cite{BHKUdensityInterior} that a piece-wise smooth density of mass in addition to the wave speeds (under certain geometric conditions) may be recovered from exterior measurements, thereby recovering all both Lam\'e parameters and the density with conormal singularities. 

 There are several works on hyperbolic inverse problems with memory; see \cite{RomanovMemory} and the references therein. In that paper, the authors consider a viscoelastic equation, which consists of the isotropic elastic wave operator plus a nonlocal operator of the form:
 \[
A(u) = \sum_{j=1}^3\int_{-\infty}^t
p(x,t-s)\delta_{ij} \nabla \cdot u(x,s) + q(x,t-s)(\p_{x_j}u_i(x,s) + \p_{x_i}u(x,s)) ds
 \]
 where $p(x,t)$ and $q(x,t)$ characeterize the viscoelasticity of the medium and are assumed to be unknown, along with the elastic moduli. In this case, the nonlocality of the operator occurs only in the time variable. They show in \cite[Theorem 3.4]{RomanovMemory} that if the density $\rho$ is known and 
$p,$ $q$ are analytic in time, then $p$, $q$, and the elastic moduli $\lambda$, $\mu$ can be uniquely determined from boundary measurements. They also assume the coefficients are sufficiently smooth. In contrast, our work does not assume that the density is known, allows for jump discontinuities in the coefficients, and considers nonlocality in the spatial variable.

The inverse problem we consider here is to recover all of the coefficients appearing in $\mL$, which have conormal singularities, from the \emph{outside measurement operator} that maps Cauchy data supported outside of $\Omega$ to the corresponding wave field restricted to $\R^n \setminus \overline \Omega$. 
 The logic of the proof proceeds as in \cite{BHKUdensityInterior} to use a layer stripping approach with scattering control to recover the singularities and ``lower order amplitudes'' in the approximate solutions to the wave equation using the FIO calculus. Thus, the novelty is quantifying how the nonlocal term affects the transport equations for the amplitudes and the recovery of the density.
 The nonlocal term presents an additional challenge in the layer stripping argument, as wave solutions inside a known layer can be influenced by interactions occurring globally within unknown layers. To address this challenge, the analysis focuses on the singularities of the solution and introduces a ``microlocal" version of the DN map or the outside measurement operator.
 
An additional challenge is the presence of the reference gravitational potential $\Phi$ and its nonlocal dependence on $\rho$. In the analysis, it behaves as an additional unknown coefficient, which is why we assume it is known outside $\Omega$. We need to recover it as part of the layer-stripping process, which couples a hyperbolic inverse problem with an elliptic inverse problem.
 
  \subsection*{Notation and statement of the main theorem}
 We refer the reader to \cite[introduction]{Montalto2014}, which describes past results for the uniqueness question on this particular inverse problem in the case where all coefficients are smooth, the operator is local and involves a general anisotropic Riemannian metric $g$. Those works involve the Dirichlet-to-Neumann map (or its inverse), but as shown in \cite{Rach00}, there is an equivalence between exterior Cauchy data and boundary data. Many other works in the literature recover lower order coefficients in a local hyperbolic equation from boundary measurements in the smooth case; see for example \cite{Montalto2014,Oksanen2020,OksanenLowerOrder,BelishevLowerOrder} and the references there.

The wave speed inside $\Omega$ is given by $c(x) = \sqrt{\frac{\kappa}{\rho}}$ defined in $\Omega$. We assume $\rho$ and $c$ are piece-wise smooth in $\Omega$ and in $C^\infty$ outside of $\Omega$. Recall that $\{\Gamma_j\}_{j=0}^{m}$ is a collection of smooth, closed hypersurfaces such that $\Gamma_0 = \partial\Omega$ and $\Gamma:= \cup_{j=0}^{m} \Gamma_j$ contains all the singularities of the parameters. Hence, all the coefficients are smooth in $\R^n$ up to $\Gamma$, with possible jumps there.

Let $T = \text{diam}_g(\Omega)$ with the rough metric $g = c^{-2}dx^2$. Also, set $\Upsilon \subset \R^n$ to be an open bounded set such that $\Omega \Subset \Upsilon$ and $\text{dist}_g(\Omega, \p \Upsilon) > 2T$. Fix a unit conormal vector to $\Gamma$ denoted $\nu = \nu(x)$ at each $x \in \Gamma$. 
Following the notation in \cite[Section 3]{SUV2019transmission},
for an interface $\Gamma$, we work locally in a small
neighborhood of a point on $\Gamma$ and call one of its sides, $\Omega_-$ negative, the other one, $\Omega_+$, positive.
For a material parameter generically denoted $c$, we have $c = c_-$ in $\Omega_-$, and $c = c_+$ in $\Omega_+$, where $c_-$, $c_+$ are smooth up to $\Gamma$ and
$c_- \neq c_+$ pointwise.
 We denote $f|_{\Gamma_\pm}$ to be the limit of $f(x)$ as $x$ approaches $\Gamma$ from the positive/negative side.
 
Define the \emph{Neumann operator} at $\Gamma_\pm$ as
\begin{equation}\label{Neumann}
\mathcal N_{\pm} u = \kappa \ \text{div}(u) \nu
- \rho \langle u, \nabla_x \Phi \rangle \nu
+ \rho \nabla \Phi  \langle u, \nu \rangle 
\restriction_{(0,T) \times \Gamma_\pm}.
\end{equation}
Note that this is not the natural geometric description for the Neumann operator on fluid boundaries but we show in section \ref{s: natural transmission} that the natural Neumann operator may be simplified to this form. 
 When we wish to leave the exact boundary restriction ambiguous and refer to the normal operator at some hypersurface $\Gamma$,  we will also write  $\mathcal N u = \kappa \text{div}(u) \nu - \rho \langle u, \nabla_x \Phi \rangle \nu
+ \rho \nabla \Phi  \langle u, \nu \rangle
 \restriction_{(0,T) \times \Gamma}$ when it is clear which hypersurface $\Gamma$ we are referring to.

Let $h=(\psi_0,\psi_1) \in H_0^1(\Upsilon)\oplus L^2(\Upsilon)$ be a set of initial data supported in $\Omega^*:=\Upsilon\setminus\overline{\Omega}$. The displacement vector of a gas or liquid parcel between the unperturbed and perturbed flow is $u(t,x)$. We consider the Cauchy initial value problem
\begin{equation}\label{IVP}
\begin{aligned}
\mL u = \rho \p_{t}^2 u - \nabla_x \kappa \nabla_x \cdot u  + \nabla_x(\rho u \cdot \nabla_x \Phi) -
(\nabla_x \cdot \rho u) \nabla_x \Phi +\rho \nabla S(\rho u)=& 0 \quad &&\mbox{in }[0,T]\times\R^n\\
(u,\p_t u)|_{t=0} =& h \quad &&\mbox{in }\R^n,
\end{aligned}
\end{equation}
and also impose the transmission conditions at each interface $\Gamma_i$ 
\begin{equation}\label{e: trans conditions}
[\nu \cdot u] = [\mathcal N u] = 0 \qquad \text{ on } \Gamma_i,
\end{equation}
where $[v]$ stands for the jump of $v$ from the exterior to the interior across $\Gamma_i$. 

To guarantee the well-posedness of this hyperbolic Cauchy problem, we will apply  \cite[Theorem 8.1]{HoopPham}.
It is easy to verify that $B(x,D)$ is symmetric and maps from $L^1(0,T, H^1_0(\Omega)$ to $L^1(0,T, L^2(\Omega)$ where we identify $B(x,D)u$ in $\R^n$ by its restriction to $\Omega$, so it may be thought of as an element in $L^2(\Omega)$. Hence, \cite[Theorem 8.1]{HoopPham} applies and so
this forward problem has a unique solution $u_h(t,x) \in C^1(\R_+, H^1(\R^n))$.
We define the outside measurement operator
\begin{equation}\label{data_map}
\mathbf{F} : h \mapsto u_h(t,x)|_{\overline \Omega^c}.
\end{equation}

In this article, we prove that one can determine all the unknown, piece-wise smooth parameters of $\mL$ from $\mathbf{F}$.
To recover the parameters in the interior, we will recover local ray transforms of various tensors involving the PDE coefficients and then use the local injectivity results \cite{UVLocalRay,SUVRigidity}. 
We need geometric conditions on the domain to obtain a global result for the parameters from local injectivity near the interfaces.
Thus, we assume that the domain $\Omega$ admits a \textit{strict convex foliation}, which is discussed in Section \ref{Sec_geometric_cond}, and appears in \cite{SUV2019transmission,CHKUElastic}.
Let there be two sets of parameters $(\rho,\kappa)$ and $(\tilde{\rho},\tilde{\kappa})$ that are piece-wise smooth in $\Omega$, and are equal to each other in $\R^n \setminus \Omega$. We also need to assume $\Phi = \tilde \Phi$ in $\R^n \setminus \Omega$.
Let us write $\tilde{f}$ to be the quantity corresponding to the set of parameters $(\tilde{\rho},\tilde{\kappa})$ where $f$ is the same quantity corresponding to $(\rho,\kappa)$.

We also need some boundedness assumptions on $\rho, \tilde \rho$ and two of its derivatives in order to apply elliptic unique continuation results in the proof of the main theorem. We assume
\begin{equation}\label{e: lispchitz rho assumption}
\rho, \rho^{-1}, \nabla_x \rho, \nabla_x^2 \rho, \tilde\rho, \tilde {\rho}^{-1}, \nabla_x \tilde \rho, \nabla_x^2 \tilde \rho \in L^\infty(\Omega)
\end{equation}

 We now state our main result.
\begin{theorem}\label{Main_th_2}
	Let $\Omega \subset \R^n$ be a bounded domain with interfaces and
	satisfies the extended convex foliation condition described in Section \ref{Sec_geometric_cond}.
	Let $(\rho,\kappa)$, $(\tilde{\rho},\tilde{\kappa})$ be two sets of piece-wise smooth parameters described above.
	If $\mathbf{F}=\tilde{\mathbf{F}}$ and $(\rho, \kappa, \Phi) = (\tilde \rho, \tilde \kappa, \tilde \Phi)$ in $\R^n \setminus \Omega$, then $\rho=\tilde{\rho}$, $\kappa=\tilde{\kappa}$ in $\overline{\Omega}$.
\end{theorem}

The novelty of the paper is that it provides a comprehensive understanding of the impact of the nonlocal self-gravitation term $S(u)$ as well as $\Phi$, which depends on $\rho$ globally, on the recovery process, emphasizing the recovery of the density and its relation to the transport equations. Essentially, we will show that $\nabla S$ and $\Phi$ affect the ``lower order amplitudes'' of the FIO parametrix for this wave equation and then analyze the effect of $\nabla S$ and $\Phi$ on the transport equations satisfied by these amplitudes to recover $\rho.$ In addition, the density $\rho$ is entangled in a single wave amplitude jointly via $\Phi$ and $S$ that depend on $\rho$ globally, so we develop an argument using propagation of singularities and Carleman estimates for an elliptic system to disentangle $\rho$ from these constituents for the recovery.
By establishing these results, the paper contributes to the study of inverse problems for PDE systems containing nonlocal operators.

Let us provide a brief outline of the paper. In Section \ref{Sec_geometric_cond}, the paper introduces the notation and outlines the geometric assumptions made for the domain $\Omega$. These assumptions are crucial for the analysis and recovery process. Additionally, the microlocal parametrix solution to the initial value problem \eqref{IVP} is constructed.

Section \ref{Sec_intro_interface_determination} presents an important proposition that plays a key role in the main theorem. This proposition enables the recovery of the jumps of the coefficients of the operator $\mathcal{L}$ across interfaces by measuring reflected waves. This step is essential in the layer-stripping argument employed in the main theorem, allowing for coefficient recovery across interfaces encountered during the process.

Finally, in Section \ref{sec: proof of main theorem}, the main theorem is proved. The section provides a detailed proof of the theorem, which demonstrates the recovery of all the unknown parameters from the outside measurement operator. The analysis takes into account the effect of the term $S(\rho u)$ on the lower-order amplitudes of the Fourier Integral Operator (FIO) parametrix for the wave equation. It is shown that $S(\rho u)$ does not affect the travel times of rays, enabling the recovery of the wave speed $c$ as in previous works. However, the terms $S(\rho u)$  and $\Phi$ does impact the transport equations satisfied by these lower-order amplitudes, leading to the recovery of the density $\rho$.
\section{Preliminaries and notation}\label{Sec_geometric_cond}
 In this section, we introduce our notation and state our geometric assumptions.
All the definitions and results in this section briefly summarize what was defined in \cite{B_density,CHKUControl,CHKUElastic}.

\subsection{Geometric Transmission Conditions} \label{s: natural transmission}

It is shown in ?? Dahlen/Tromp equation (3.80) that the fluid Neuamnn operator is
\begin{equation}\label{e: trans geometric neumann}
\mathcal N u
= p^0(\gamma-1) (\nabla \cdot u) \nu
+p^0 (\nabla u) \cdot \nu+\nu \nabla_{\Gamma}\cdot (p^0 u)
- p^0 (\nabla_\Gamma u) \cdot \nu |_{\Gamma}
\end{equation}
where $p^0$ is the hydrostatic pressure solving \[\nabla p^0 = -\rho \nabla \Phi,\]
$\nabla_\Gamma$ is the tangential del operator with respect to the hypersurface $\Gamma$ and $\gamma = \kappa/p^0$ is the adiabatic index. The quantity $\mathcal N u$ has to be continuous across each interface along with the tangential slip transmission condition
\[
[\langle u, \nu \rangle]^+_- = 0.
\]
To facilitate the microlocal parametrix construction, we need to simplify $\mathcal N u$ to isolate the first order differential part from the remaining components.
Note that
\[
\nabla u \cdot \nu
= \nabla_\Gamma u \cdot \nu +
\langle \nabla_\nu u, \nu \rangle \nu. 
\]
By plugging this into \eqref{e: trans geometric neumann}, we get a cancellation with the last term in $\mathcal N u$
so 
\[
\mathcal N u
= p^0(\gamma-1)(\nabla \cdot u)\nu 
+ p^0\langle \nabla_\nu u, \nu \rangle \nu
+ \nu \nabla_\Gamma \cdot (p^0 u).
\]
Next, we expand
\[
\nabla_\Gamma
\cdot (p^0 u)
= \nabla \cdot (p^0 u)
- (\nabla_\nu p^0 u) \cdot \nu
\]
\[
=p^0 \nabla \cdot u
+ \langle\nabla p^0, u \rangle
- (\nabla_\nu p^0) \langle u, \nu \rangle
- p^0 \langle \nabla_\nu u, \nu \rangle. 
\]
Plugging this in gives
\[
\mathcal N u
= p^0 \gamma \nabla \cdot u \nu 
+ \nu \langle \nabla p^0, u \rangle
- \nabla_\nu p^0 \langle u, \nu \rangle \nu 
\]
So using the tangential slip condition, the Neumann operator becomes
\begin{equation}\label{e: trans final Neumann op}
\mathcal N u
= 
\kappa (\nabla \cdot u) \nu
- \nu \rho \langle \nabla \Phi, u \rangle 
+ \rho \langle \nabla \Phi, \nu \rangle \langle u, \nu \rangle \nu.
\end{equation}
As derived in \cite{HoopPham}, at the interface, $\nabla \Phi$ is parallel to  $\nu$ that
so
\[
\langle \nabla \Phi, \nu \rangle \nu = \nabla \Phi 
\]
and this gives the transmission conditions \eqref{e: trans conditions}.
\subsection*{Geometric assumptions}
The principal symbol of the hyperbolic operator $\mL$ is given as
\begin{equation}\label{princ_symb}
p(t,x,\tau,\xi)
= -\rho (\tau^2 - c^2 \xi \otimes \xi).
\end{equation}
It is easy to check that the eigenvalues of the principal symbol are $\rho \tau^2$ and $\rho(\tau^2 - c^2 |\xi|^2)$. Hence, the characteristic set for $\mL$  that determines the ray geometry is
\[\Sigma_p:=\{\text{det}(p(t,x,\tau,\xi)) = 0\}
= \{ \tau^2-c^2|\xi|^2 = 0
\text{ or } \tau = 0 \}\].

One can calculate the lower order terms in the full symbol of $\mL$, and the order one term is

\begin{equation}\label{sub-princ_symb}
p_1(t,x,\tau,\xi) = i\nabla_x\kappa \otimes \xi -  i \rho \xi \otimes \nabla_x \Phi 
+ i \rho \nabla_x \Phi \otimes \xi 
\end{equation}

\subsection*{Geodesics and bicharacteristics}
The bicharacteristics curves $\gamma^{\pm}$ in $T^*(\R\times\Ups)$ are the integral curves of the Hamiltonian vector fields $V_{H^{\pm}}$, where $H^{\pm} = \tau \pm c|\xi|$ along with the condition that $\gamma^{\pm}$ lies in the set $\Sigma_p$.
Parametrized by arclength $s\in \R$, one obtains
\begin{equation}\label{parametrization_bichar}
\frac{dt}{ds} = c^{-1}, \qquad
\frac{dx}{ds} = \pm \frac{\xi}{|\xi|}, \qquad
\frac{d\tau}{ds} = 0, \qquad
\frac{d \xi}{ds} = \mp |\xi| \nabla_x (\log c),
\end{equation}
along with the condition that $p(t,x,\tau,\xi) = 0$.
We refer to $\gamma^{\pm}$ as the forward and backward bicharacteristic curves.

\subsection*{Foliation Condition}\label{s: foliation condition}

We assume that the domain $\Omega$ has an \emph{extended convex foliation} with respect to the metrics $g_{c}$. This is an extension of the convex foliation condition given in \cite{UVLocalRay} to the piece-wise smooth setting and was introduced in \cite[Definition 3.2]{CHKUElastic}.
\begin{Definition}[Extended convex foliation]\label{foliation_defn}
	We say  $\sig : \overline{\Omega} \mapsto [0,\tilde{\T}]$ is a (piece-wise) extended convex foliation for $(\Omega,g_{c})$ if
	\begin{enumerate}
		\item $\sig$ is smooth and $d\sig \neq 0$ on $\overline{\Omega}\setminus\Gamma$.
		\item $\sig$ is upper semi-continuous.
		\item each level set $\sig^{-1}(\T)$ is geodesically convex with respect to $g_c$, when viewed from $\sig^{-1}((\T,\tilde \T))$ for any $t \in [0,\tilde{\T})$.
		\item $\partial\Omega = \sig^{-1}(0)$ and $\sig^{-1}(\tilde{\T})$ has measure zero.
		\item there is some $\T_i \in [0,\tilde{\T}]$ such that $\Gamma_i \subset \sig^{-1}(\T_i)$ for $i=0,\dots,m$.
		\item $\limsup_{\epsilon \to 0+}c|_{\sig^{-1}(\T+\epsilon)}
		\leq \limsup_{\epsilon \to 0^+} c|_{\sig^{-1}(\T-\epsilon)}$ whenever $\Gamma_i \subset \sig^{-1}(\T)$ for some $i$ and $\Gamma_i$.
	\end{enumerate}
We say $c$ satisfies the \emph{extended foliation condition} if there exists an extended convex foliation for $(\Omega, g_c)$.
\end{Definition}

See the discussion below \cite[Definition 3.2]{CHKUElastic} for an explanation of the last condition. We write $\Omega_{\T}:= \sig^{-1}(\T,\tilde{\T}]$ to denote the part of the domain whose boundary is $\Sigma_\T:= \sig^{-1}(\T)$.
Let us fix the convention of writing `above' $\Sigma_\T$ to be outside of $\Omega_\T$ and `below' to be inside $\Omega_\T$.
We write $\Sigma^{\pm}_\T$ to denote two copies of $\Sigma_\T$ approached from 'above' or `below' $\Sigma_\T$.
Let us observe that, if required, we can extend each $\Gamma_j$ along with $\Sigma_{\T_j}$ and denote $\Omega_j$ to be the connected components of $\overline{\Omega}\setminus \Gamma$. Write $\widetilde{\Omega}_j$ to be $\Omega_{\T_j}$, where $\Gamma_j\subset\Sigma_{\T_j}$. The boundary distance function of the domain $\Omega_\T$ will be denoted $d^{\T}(\cdot, \cdot)$.

We end this section by summarizing all the assumptions we have made so far.
\begin{Assumption}
	With the above notation, we assume that:
	\begin{enumerate}
		\item the interfaces $\Gamma_j$ are a collection of disjoint, connected, closed hypersurfaces in $\overline{\Omega}$,
		\item $\Omega$ has an extended convex foliation $\sig$ with respect to $g_c$.
		
	\end{enumerate}
\end{Assumption}

Next, we will describe the microlocal parametrix for the wave equation and a key proposition that lets us generate specified virtual sources in the interior from the outside measurement operator.
\subsection{FIO parametrix and scattering control}\label{s: FIO parametrix}

In this section we construct the microlocal parametrix solution to \eqref{IVP} in the smooth setting similar to the one in \cite{RachBoundary}. Since the arguments we will use to prove the main theorem are done locally, this will suffice. The full parametrix with transmission conditions where the material parameters have discontinuities is in \cite{CHKUUniqueness}.
First, the geometric optics solution for the initial value problem \eqref{IVP} when the material parameters are smooth has the form
\begin{equation*}
U = E_0 \psi_0 + E_1 \psi_1,
\end{equation*}
where $E_k$, $k=0,1$ are the solution operators given in terms of FIOs. We write $A\equiv B$ for two FIOs $A$ and $B$ to denote that $A$ is same as $B$ modulo a smoothing operator.
We impose that the FIOs $E_0$ and $E_1$ solve the following system modulo a smoothing operator.
\begin{equation}\label{E_k PDE}
\begin{aligned}
\mathcal L E_{\k} &\equiv 0 \quad &&\mbox{on } (0,\infty) \times \R^n,\\
E_0|_{t=0} &\equiv I, \quad \partial_t E_0|_{t=0} \equiv 0, \quad&& \mbox{on }\R^n,\\
E_1|_{t=0} &\equiv 0, \quad \partial_t E_1|_{t=0} \equiv I, \quad&& \mbox{on }\R^n.
\end{aligned}
\end{equation}
The FIOs $E_{\k}$ for ${\k}=0,1$ are given as
\begin{equation}\label{E_k FIO}
\begin{aligned}
E_{\k} f
=&\sum_{\pm} \int_{\R^n} e^{i\phi^{\pm}(t,x,\eta)} e_{\pm,\k}(t,x,\eta) \hat{f}(\xi)\,d\xi
+
C_{\pm, \k}(t,x,D_x)f,
\end{aligned}
\end{equation}
where $C_{\pm, \k}$ is a classical $\Psi$DO, $e_{\pm,k}(t,x,\xi)$ is the amplitude given as a classical $3 \times 3$ matrix symbol, and $\phi^{\pm}$ is the phase function.

The phase functions $\phi^{\pm}(t,x,\xi)$ is homogeneous of order 1 in $|\xi|$ and solve the eikonal equations $\text{det} \  p(t, x, \p_{t,x}\phi^\pm) = 0$ where $\p_t \phi \neq 0$, namely
\[
\p_t \phi^\pm = \mp c |\nabla_x \phi^\pm|.
\]
The eikonal equations are noncharacteristic, first-order nonlinear equations which
can be solved by Hamilton-Jacobi theory with initial value
\[
\phi^\pm |_{t = 0} = x \cdot \eta
\]
given in terms of the parameter $\eta \in \R^3 \setminus 0$. It follows that
\[
\nabla_x \phi^\pm |_{t=0} = \eta.
\]
We also denote 
\[
N(t,x,\eta) = \frac{\nabla_x \phi}{|\nabla_x \phi|}
\]
where we leave out the $\pm$ index when it is clear which phase function is being used.

The amplitude $e_{\k}$ has an asymptotic expansion
\[
e_{\k} \sim \sum_J (e_{\k})_J
\]
where $(e_{\k})_J$ is homogeneous of order $|J|$ in $|\eta|$.
The $m$'th column $(e^{\cdot,m}_{\k})_J$ has the form
\[
(e_{\k}^{\cdot,m})_J
= (h^{\cdot,m})_J
+ (\alpha_{\k}^m)_J N
\]
where $(\alpha_{\k}^m)_J(t,x,\eta)$ are scalars that will satisfy certain transport equations and are homogeneous of degree $J$ in $\eta$, and $(h^{\cdot,m})_0= 0$, 
$(h^{\cdot,m})_J$ lie in the cokernel of $p(t,x,\p_{t,x}\phi)$. Thus, it is sometimes convenient to write
\[
e_{\k}^{\cdot,m}
= h^{\cdot,m}
+ \alpha_{\k}^m N
\]
with $ h^{\cdot,m} \sim \sum_{J =-1, -2, \dots} ( h^{\cdot,m})_J $ and 
$\alpha_{\k}^m \sim \sum_{J =0,-1,-2,\dots}(\alpha_{\k}^m)_J$.

Let us assume we are dealing with the case $\k = 0$ and $+$ since the other cases are analogous so we leave out these indices.
Each $(\alpha^m)_J$ satisfies a first order transport equation along the bicharacteristic with Hamiltonian $H = \tau - c|\xi|$ so there is a unique choice, at $t =0$, of initial conditions in order that the initial conditions for $E_{\k}$ are satisfied. In the case $J = 0$, for example, we have
\[
(\alpha^m)_0 = \frac 1 2 \frac{\eta_m}{|\eta|}
\]
It follows that the principal term
$(e_+)_0$ (at $t=0$) of $e_+(0,x,\eta)$ is given by
\[
(e_+)_0 = 
\left( (\alpha^1)_0 N, 
\ (\alpha^2)_0 N, 
\  (\alpha^3)_0 N \right)
 = \frac{1}{2} \frac{\eta}{|\eta|} \otimes \frac{\eta}{|\eta|}
 = \frac{1}{2}\text{Proj}_\eta
\]
where $\text{Proj}_\eta$ is the projection onto the space spanned by $\eta$. Similarly, we may verify that the principal symbol for $C_\pm$ denoted $(c_\pm)_0$ will be $\frac 1 2 \text{Proj}_{\eta^\perp}$ which is the projection onto the space orthogonal to $\eta.$

Hence, to simplify the notation going forward, since $(\alpha^m )_J$ satisfy the same transport equation for each $m$ and only the initial condition depends on $m$, we will merely analyze the quantity 
\[
\mL(t,x,D_{t,x}) e^{i\phi}(\alpha N+ h_{-1}), 
\]
and set
\[
a = \alpha N,
\]
treating $\eta$ as a parameter,
in order to construct the parametrix, with the $m,\pm,\k$ indices ignored, so that the full matix parametrix may be assembled with the previous analysis, without explicitly writing it out.

Also, the outoing/incoming parametrix mapping boundary data to the solution will have a similar construction where locally, if $x_3$ represents the defining function at the boundary, with coordinates $(t,x',x_3)$, then we may view it as a Cauchy data problem as above where $x_3$ is in place of the time variable. In this case, we have the above construction with a different initial conditions and parameters. Namely, the symbols have the form $e_\pm = e_{\pm}(t,x',x_3, \eta', \tau)$ and similarly for the other quantities.

\subsection*{Transport equations}
As mentioned above, we may construct a parametrix that propagates boundary data specified at $\Sigma_{\T}$ into the interior. We start with the ansatz
\begin{equation}\label{wave_parametrix_Interior}
u(t,x,\tau,\eta') = e^{i\varphi(t,x,\tau,\eta)}a(t,x,\tau,\eta')
\end{equation}
where $a = \alpha N+ h_{-1}$ above.
Note again that we have suppressed the $\k, \pm, m$ indices whenever it is clear from the context. 
The phase function $\varphi$ is given as a solution of the Eikonal equation
\begin{equation}\label{Eikonal_eq_Interior}
|\p_t \varphi|^2 = c^2(x)|\nabla_x\varphi|^2,
\quad \mbox{with the condition}\quad
\varphi|_{x \in \Sigma_{\T}} = -t\tau + x'\cdot\eta'.
\end{equation}

From the condition $\mL u = 0$, the amplitudes solve the algebraic equation
\begin{equation}\label{e: symbolic version Pu=0}
p(t,x, \p_{t,x}\phi)(a)_{J-1}
= B_p(a)_J + C_p(a)_{J+1}+ (c)_{J+1}
\end{equation}
where $p(t,x,\tau,\xi)$ is the principal symbol of $\mL$, $p_1(t,x,\tau,\xi)$ is the sum of the lower order terms in the symbol of $\mL$, 
\begin{multline}\label{e: B_p term}
B_p M
= i\p_{\tau,\xi}p(t,x,\p_{t,x}\phi) \cdot \p_{t,x} M
+i(ip_1)(t,x,\p_{t,x}\phi_p)M
\\
+\frac{i}{2}\sum_{|\alpha|=2}\sum_{l=1}^3
(\p_{\tau,\xi}^\alpha p^{il})(t,x,\p_{t,x}\phi)
\cdot (\p^\alpha_{t,x} \phi) M^l,
\end{multline}
\begin{multline}\label{e: C_p term}
C_p M = \p_{\tau,\xi}(ip_1)(t,x,\p_{t,x}\phi) \cdot \p_{t,x}M \\
+\frac 1 2 \sum_{|\alpha|=2}\sum_{l=1}^3 (\p^\alpha_{\tau,\xi}p^{il})(t,x,\p_{t,x}\phi)
\cdot \p_{t,x}^\alpha M^l, 
\end{multline}
and $(c)_{J+1}$ is determined by Lemma \ref{l: asymp expansion of P applied to lagrangian distribution} with $P$ there given by $B(x,D)$ in our setting with principal symbol $b_0(x,\xi)$. Note that lemma \ref{l: asymp expansion of P applied to lagrangian distribution} and the above formula is valid only in $\R^3 \setminus \Gamma$, where the coefficients are smooth. The FIO construction at $\Gamma$ is done separately in the appendix.

In particular
\begin{equation}
(c)_0 = \rho^2(x) b_0(t,x,\p_x\varphi)(\alpha)_0, \qquad x \in \R^3 \setminus \Gamma
\end{equation}
and, for ease of notation, set $(c)_1 = 0, \ (a)_1 = 0$.
By applying the operator $\text{Proj}_N$ to \eqref{e: symbolic version Pu=0} we get the compatibility condition
\begin{equation}\label{transport_eq_Interior}
N\left[ 
B_p(a)_J + C_p(a)_{J+1} + (c)_{J+1}
\right] = 0
\qquad J = 0, -1, -2, \dots
\end{equation}

This reduces to a transport equation form $(\alpha)_J$ along a bicharacteristic.
We denote $\zeta_{\pm}$ to be bicharacteristics of the wave operator $\mL$ parametrized by $s$ and satisfying \eqref{parametrization_bichar}.
From \eqref{transport_eq_Interior} and setting $J=0$ we conclude via \cite[section 4]{RachDensity} that the amplitudes $(\alpha)_0$ satisfies the transport equations on $\zeta_\pm$ given as
\begin{equation}\label{transport_eq_J}
\frac{d}{ds}(\alpha)_0 = 
-\left[\frac{d}{ds} \log \sqrt{\rho c}
- \frac{1}{2} (\nabla_x \cdot N))
\right](\alpha)_0, 
\end{equation}
where no $\Phi$ term appears due to a cancellation. This is because the first order symbol is $p_1(x,\xi) = -i \nabla_x \kappa \otimes \xi + i \rho \xi \otimes \nabla_x\Phi - i \rho \nabla_x \Phi \otimes \xi$
but $N \cdot p_1 N$ is independent of $\Phi$.
 
 Thus, the solution to this ODE restricted to a boundary or interface does \emph{not} depend on $\rho$ in the interior, and can be recovered just from $c$ in the interior and $\rho$ at the boundary. Also, it is unaffected by the self-gravitation term $\nabla S(\rho u)$. Hence, in order to recover $\rho$ in the interior, one needs to analyze $(\alpha)_{-1}$.
The ODE for $(\alpha)_{-1}$ will be of the form
\begin{equation}
\frac{d}{ds}(\alpha)_{-1} = 
-\left[\frac{d}{ds} \log \sqrt{\rho c}
- \frac{1}{2} (\nabla_x \cdot N)) 
\right](\alpha)_{-1}
+ G
\end{equation}
where will later compute $G$ explicitly, and it will depend on the principal symbol of $\rho \nabla S \rho$ in the interior, as well as $\Phi$. We refer to $(\alpha)_{-1}$ as the lower order amplitude and we will show that it depends on $\rho$ in the interior, enabling the interior determination of the density.

As seen above, having the nonlocal term with $B(x,D)$ requires a modification of the usual compatibility conditions. It cannot be ignored since the density recovery in the interior will depend on $(\alpha)_{-1}$, which depends on the additional terms involving $\rho$, namely $\nabla S(\rho u)$ and $\Phi$.
We present an important lemma from \cite[Lemma 3.3]{Oksanen2020} adapted to our setting with an analogous proof.
\begin{Lemma}\label{l: asymp expansion of P applied to lagrangian distribution}
Let $P(x,D) \in \Psi^m$ with proper support and a classical symbol. Let $a(x,\eta)$ be a zeroth order symbol with asymptotic expansion $a \sim \sum_{J=0}^\infty a_{-J}$ where $a_{-J} \in S^{-J}_{hom}$. Let $\Phi(x,\eta)$ be a phase function. Then
\[
e^{-i\Phi }Pe^{i\Phi} a
\sim \sum_{J = 0, -1, -2, \dots} c_{m+J}, \qquad c_{m+J}(x,\eta) \in S^{m+J}
\]
where
\[
c_m = p_m(x, d_x \Phi)(x,\eta)) a_0(x,\eta)
\]
\[
c_{m-1} = p_m(x, d_x \Phi) a_{-1}
+ \frac{1}{i}L a_0 + b(x,\eta) a_0,
\]
where in local coordinates
\[
Lu = \p_\xi p_m(x, d_x\Phi) \cdot \p_x u
\]
and
\[
b(x,\eta)= \frac{1}{2i} \p_{\xi_j \xi_k} p_m(x,d_x \Phi) \p_{x_j x_k}\Phi(x,\eta) + p_{m-1}(x,d_x \Phi(x,\eta))
\]
\end{Lemma}
Applying the above lemma to operator $B$ replacing $P$ and $m=0$, and taking $u$ as in \eqref{wave_parametrix_Interior}, 
the transport equation of is \eqref{transport_eq_Interior} (where we set $c_1 = 0$). Hence, we obtain the compatibility condition \eqref{e: symbolic version Pu=0}.

\begin{Remark}[Quantifying the effect of gravitation on the recovery procedure]
With the above notation, our proof will show that $\alpha_0 N$ is independent of $\rho, \Phi,$ and $S(\rho u)$. Hence, $\rho$ must be recovered in the interior from the lower order amplitude $\alpha_{-1}$ instead, which also depends on the gravitation terms $\Phi$ and $S(\rho u)$, and one is often interested in quantifying the effect of these gravitational terms on the amplitude of waves. This problem has a long history.
As described in \cite[Chapter 1]{DahlTromp}, in terrestrial seismology, the gravitational terms in the elastic equation are often excluded since observed oscillations of the Earth are predominantly elastic. It is also possible to ignore the first-order perturbation in the gravitational potential (represented by the term $S(\rho u)$) while retaining the initial potential $\Phi$, which is known as the \emph{Cowling approximation}. This permits both cases of predominantly gravitational as well as predominantly elastic oscillations while also retaining a local wave operator.

The justification usually given for ignoring the gravitational terms in the elastic equation is that their effect is negligible at the high frequencies typically studied, and their presence is only detected at low frequencies. 
This reasoning is only true when studying the leading order (in terms of powers in frequency, or Sobolev order of the wave) amplitude of a wave. 
However, the next-to-leading order amplitude represented by $\alpha_{-1}$ is needed to recover density, and the Cowling approximation is no longer valid for this term since the gravitational contribution is of the same order as the elastic contribution to this constituent of the wave. Practically, these gravitational ripples are small, but this paper gives an exact theory on doing the density recovery. 
\end{Remark}

An important fact is that since $B$ is a zeroth order operator, the transport equation for $(\alpha)_0$ remains the same as in \cite{RachBoundary} where density does not appear. Our first step is to show that we may recover the parameters and their derivatives at the interfaces from the knowledge of the solution on a one-sided neighborhood of the interface. This will later allow us to recover $q$-interior travel times past an interface in order to use the local geodesic ray transform results in \cite{UVLocalRay}.

\section{Interface determination}\label{Sec_intro_interface_determination}
 The first crucial piece for proving the main theorem is to recover the parameters and their derivatives across interfaces. Since we will prove the main theorem via a ``layer-stripping'' argument using the foliation function, we initially recover the coefficients in the first layer. In order to begin recovery of the second layer, we must show that all coefficients and their derivatives may be recovered across the first interface $\Gamma_1$ from $\mathbf F$. This will follow if we can show that this recovery may be made from the reflection operator $M_R$, which is a constituent of the parameterix of the wave equation, and it is the primary result of this section.
We prove that knowledge of the data only on one side of the interface is enough to recover the parameters on both sides of the interface.
 Later, we will only require a local recovery in the vicinity of a single interface, so our interface determination result is stated with only one interface $\Gamma$ and a small neighbourhood $\Oo$ of $\Gamma$ in $\R^n$.

Let $u_I$ be an incoming wave, started outside $\Oo$ and approaching $\Gamma$. After hitting $\Gamma$ transversely at some point $z_0$, $u_I$ splits into two parts: the reflected wave $u_R$ and the transmitted wave $u_T$. We will now construct an FIO parametrix of the solution in an open neighborhood of $z_0$ as done in \cite{SU-TATBrain,BHKUinterface}.
The set $\Gamma$, after a suitable extension, splits $\Oo$ into two parts $\Oo_{\pm}$ for two sides of $\Gamma$, where $u_I$ and $u_R$ travels through $\Oo_{-}$ and $u_T$ travels through $\Oo_{+}$.
We denote $\Gamma_{\pm}$ be two copies of $\Gamma$ approached from $\Oo_{\pm}$. From here onwards, we write $f_{\pm}$ to denote $f$ on $\bar{\Oo}_{\pm}$. For a fixed $h \in \mathcal E'(\R_t \times \Gamma)$ microsupported away from the glancing set, each constituent can be represented using an FIO in the form
\begin{equation}\label{e: u_I,R,T FIO rep}
u_{\bullet}(t,x) = \int e^{i\varphi_{\bullet}(t,x,\tau,\eta')}e_{\bullet}(t,x,\tau,\eta')\hat{h}(\tau,\eta')\,d\tau\,d\eta',
\qquad \bullet = I,R,T,\end{equation}
where $\varphi_{\bullet}$ is the phase function and $e_\bullet$ has an asymptotic expansion
\begin{equation*}
e_\bullet(t,x,\tau,\xi) = \sum_{J=0,-1,-2,\dots} (e_\bullet)_J(t,x,\tau,\xi'),
\end{equation*}
where $(e_\bullet)_J$ is homogeneous of order $|J|$ in $|(\tau,\xi')|$ (see \cite{SU-TATBrain} for more details).

Similar to the construction in section \ref{s: FIO parametrix}, each column of $e_\bullet$, denoted $e^m_\bullet$, is of the form
\[
e^m_\bullet = \alpha^m_\bullet N_\bullet + h^m_{\bullet, -1}
\]
with the same definitions, except that we treat we treat $\tau, \eta'$ as the parameter. The compatibility conditions remain the same as before, so that $(\alpha^m_\bullet)_J$ satisfies the same transport equations. The only change is that the initial conditions for the transport equations are determined at $\R_t \times \Gamma$, and assuming that the bicharacteristic is non-glancing. 

The transmission of the wave obeys Snell's law of refraction for pressure waves. At the interface $\Gamma$, the reflected and the transmitted waves satisfy a set of transmission conditions discussed in Section \ref{Sec_reflection_transmission}. The solution near $z_0$, denoted $u(t,x)$, may be written as $u = u_I + u_R$ in $\Oo_-$ and $u = u_T$ in $\Oo_+$.
From the transmission conditions, one can calculate the transmission and reflection operators $M_R$ and $M_T$, which are zeroth order $\Psi$DOs on $\Gamma$ belonging to the set $\Psi_{cl}^0(\Gamma)$, so that
\[ u_R|_{\Gamma_{-}} = M_R h \quad \mbox{and}\quad u_T|_{\Gamma_{+}} = M_T h, \]
where $h:= u_I|_{\Gamma_{-}}$, and $u$ satisfies the transmission conditions. We calculate the operators $M_{R/T}$ explicitly in Section \ref{Sec_reflection_transmission}. With this notation, we state the main result of this section.
\begin{Proposition}\label{Main_th_1}
	Let $\Oo$ and $\Gamma$ be as above. Suppose that $\Sigma_q \subset \Gamma$ and $c = \tilde c, \rho = \tilde \rho$
 in $\Oo_-$, and
 \[
  M_R \equiv \tilde M_R \text{ on } T^*\Sigma_q^-
 \]
 for some time $T>0$. Then $\p_{\nu}^j c^{(+)} = \p_{\nu}^j\tilde c^{(+)}$ and  $\p_{\nu}^j\rho^{(+)} = \p_{\nu}^j\tilde \rho^{(+)}$ on $\Gamma$ for all $j = 0, 1, 2, \dots$ and $M_T \equiv \tilde M_T$ on $\Gamma$.

\end{Proposition}

The proof does not simply follow from the arguments in \cite{BHKU_1} since we have additional terms in the transmission conditions which need to be disentangled.
We leave the proof to the appendix but compute the principal symbol of $M_R$ and $M_T$, and verify that it is elliptic under the right conditions. In addition, the symbols of $M_R$ will be diagonal since we showed in section \ref{s: FIO parametrix} that $\mathcal N u$ is a multiple of $\nu$, despite the extra terms from the gravitational field.

To compute the principal symbol of $M_R$, we take 
\[
u_\bullet = e^{i\phi_\bullet} \alpha_\bullet N_\bullet
\]
and $u = u_I + u_R + u_T$, each $u_\bullet$ being supported either in $\mathcal O_+$ or $\mathcal O_-$. For convenience, we use local coordinates such that $\Gamma = \{x_3 = 0\}$. Denote $\xi_{I,3} = \sqrt{c_-^{-2}\tau^2 - |\eta'|^2}$, $\xi_{R,3} = -\xi_{I,3}$
and $\xi_{T,3} = \sqrt{c_+^{-2}\tau^2 - |\eta'|^2}$. Also $\hat \xi_{\bullet,3} =\xi_{\bullet,3}/|\eta'|^2 $.
Next, we apply the transmission conditions \ref{e: trans conditions}.
First, $[u \cdot \nu ]^+_- = 0$ gives 
\begin{align*}
&[\alpha_I N_I + \alpha_R N_R] \cdot \nu
= \alpha_T N_T \cdot \nu
\\
&(\alpha_I - \alpha_R) \langle N_I,\nu \rangle = \alpha_T \langle N_T, \nu \rangle
\\
&(\alpha_I - \alpha_R) \hat \xi_{I,3} = \alpha_T \hat \xi_{T,3}
\end{align*}
Applying the second transmission condition $[\mathcal{N}u]^+_-=0$ to leading order terms, we obtain
\begin{align*}
\nu \kappa_- i (\nabla \phi_I \cdot N_I)\alpha_I
+ \nu \kappa_- i (\nabla \phi_R \cdot N_R)\alpha_R
&= \nu i \kappa_+(\nabla \phi_T \cdot N_T) \alpha_T
\\
\mbox{and} \qquad
\nu \kappa_- i|\xi_I|\alpha_I + \nu \kappa_- i |\xi_R| \alpha_R
&= \nu \kappa_- |\xi_{T,3}| \alpha_T,
\end{align*}
where $N_\bullet$ is the Neumann operators \eqref{e: trans geometric neumann} corresponding to $\bullet=I,R,T$.
Thus, the transmission conditions to leading order may be written as a matrix equation
\[
\bmat{\hat \xi_{I,3} & \hat \xi_{T,3} \\
-\kappa_-|\xi_I| & \kappa_+ |\xi_T|}\col{(\alpha_R)_0 \\ (\alpha_T)_0}
= \col{\hat \xi_{I,3} \\
\kappa_+ |\xi_I|} (\alpha_I)_0
\]
and after some algebra
\[
(\alpha_R)_0
= \frac{\xi_{I,3}\rho_+ - \xi_{T,3}\rho_-}{
\xi_{I,3}\rho_+ + \xi_{T,3}\rho_-
}(\alpha_I)_0
\qquad (\alpha_T)_0 
= \frac{2\kappa_- \xi_{I,3}}
{\kappa_+\frac{c_-}{c_+}\xi_{I,3} + \kappa_- \frac{c_+}{c_-}\xi_{T,3}} (\alpha_I)_0
\]
which gives the principal symbols
\begin{equation}\label{e: princ symbol M_R and M_T}
\sigma_0(M_R)(t,x',\tau,\eta')
= \frac{\xi_{I,3}\rho_+ - \xi_{T,3}\rho_-}{
\xi_{I,3}\rho_+ + \xi_{T,3}\rho_-
} \otimes I, \qquad \sigma_0(M_T)(t,x',\tau, \eta')
= \frac{2\kappa_- \xi_{I,3}}
{\kappa_+\frac{c_-}{c_+}\xi_{I,3} + \kappa_- \frac{c_+}{c_-}\xi_{T,3}} \otimes I. 
\end{equation}
Note that this construction of $M_R$ on $\Gamma_-$ and $M_T$ on $\Gamma_+$, and a similar construction can be made for the opposite sides of the interface.
Hence $M_T$ is elliptic no matter the values of $\rho$ and $c$. However, we see that $M_R$ may fail to be elliptic, but let us show that this can happen at only one particular incidence angle, which we can just avoid in our scattering control construction. 

If $\sigma_0(M_R) = 0$, then $\xi_{I,3}\rho_+ =\xi_{T,3}\rho_-$ and after factoring out $\tau$, denoting $\hat \eta' = \eta'/\tau$, we get
\[
\rho_+\sqrt{c_-^{-2} - |\hat \eta'|^2} =  \rho_-\sqrt{c_+^{-2} - |\hat \eta'|^2}.
\]
Solving for $|\hat \eta'|$ we obtain
\begin{equation}\label{e: Brewster angle}
|\hat \eta'|
= \sqrt{\left|\frac{(\rho_+c_+)^2-(\rho_-c_-)^2}{(c_+c_-)^2(\rho_+^2-\rho_-^2)}\right|},
\end{equation}
which is an analog of the Brewster angle in light scattering. The lower order terms in the classical expansion of the symbol of $M_{R/T}$ is in the appendix along with the proof of the proposition. Hence, for $\hat \eta'$ satisfying either the equation above or in a glancing set, we cannot make the scattering control parametrix analogous to the one in \cite{CHKUElastic} so we will avoid these bad codirections. It is convenient to denote this degenerate set as
\[
\mathcal D := \{ (t,x,\tau,\eta) \in T_{\R \times \Gamma}^*(\R \times \Omega) \cap \Sigma:
c_-^{-2} - |\hat \eta'|^2 = 0 \text{ or } c_+^{-2} - |\hat \eta'|^2 = 0 \text{ or } \hat \eta' \text{ satisfies } \eqref{e: Brewster angle} \}
\]

\section{Interior determination and proof of the main theorem} \label{sec: proof of main theorem}
In this section, we prove the uniqueness of the piecewise-smooth parameters $\rho$ and $\mu$, in the interior of a domain $\Omega\subset\R^n$, $n\geq 3$, with disjoint smooth interfaces. We start with a key proposition that lets us generate virtual sources in the interior of $\Omega$ using the outside measurement map.

\begin{Proposition}\label{Prop SC tail}
	For a given distribution $v$ with $WF(v) = \{(x,r\xi) :r\in\R_{+} \}$ for some $(x,\xi) \in S^*\Omega_{\tau}$, there exists a Cauchy data $h_{\infty}$ supported in $\Upsilon \setminus\overline{\Omega}$ such that, for a large time $T>0$, $WF(\mathcal{F}_{T+s}h_{\infty}-\mathcal{F}_s v) = \emptyset$, for any $s\geq 0$.
\end{Proposition}

\begin{proof}
This proof is identical to the one in \cite[Proposition 5.3]{CHKUElastic}, which was done in the elastic case. The proof only requires the extended convex foliation condition and that the operators $M_T, M_R$ are elliptic in the non-glancing sets. The set of allowable codirections there where the proposition holds was denoted $\mathcal S$ and it is dense within a certain set described in \cite[lemma 5.9]{CHKUElastic}. By replacing $\mathcal S$ with $\mathcal S \setminus \mathcal D$, that lemma continues to hold since $M_T$ and $M_R$ are elliptic at each point in $\mathcal S \cap T^*_{\R \times \Gamma}(\R \times \Gamma)$ modulo a measure zero set. Since $B(x,D)$ and the $\Phi$ terms does not affect the principal symbol of $\mL$, the proof in \cite[Proposition 5.3]{CHKUElastic} goes through without change.
\end{proof}

\subsection{Proof of Theorem \ref{Main_th_2}}

The structure of the layer stripping proof follows \cite{BHKUdensityInterior}, but we take extra care to deal with the new terms that involve $\Phi$ and $S(\rho u)$.
To recover the density $\rho$, we dig deep into the transport equations along the bicharacteristic curves of $\mL$ and using local injectivity of the geodesic ray transform \cite{UVLocalRay,SUVlocaltensor} together with the elliptic equations that $S(\rho u)$ and $\Phi$ satisfy, we stably determine the coefficients in a subset of $\Omega_{\T}$ near $\Sigma_{\T}$.
We proceed deeper inside $\Omega$ by an iterative method and using the strict convex foliation of the domain $\Omega$.

\begin{proof}[Proof of Theorem \ref{Main_th_2}]
The proof is by contradiction. Suppose $c\neq \tilde c$ or $\rho \neq \tilde \rho$, and let $f= |c - \tilde c|^2 + |\rho - \tilde \rho|^2 + |\Phi- \tilde \Phi|^2 $. Now consider $S:= \Omega \cap \text{supp} f$, and take $\T = \min_S \sig$: 
so $c = \tilde c$,  and $\rho = \tilde \rho$ and $\Phi = \tilde \Phi$ above $\Omega_{\T}$, but by compactness there is a point $x \in \Sigma_{\T} \cap S$. The condition that $\sig^{-1}(\tilde {\T})$ has measure zero rules out the trivial case $\T = \tilde \T$. The outside measurement map lets us recover wave fields outside $\Omega$, so the first part of the proof is to show that we can recover wave fields
outside $\Omega_{\T}$.

Let us now consider a small neighborhood of $x$, denoted $B_x$, and we consider the $\T$-boundary distance function $d^{\T}$ restricted to this neighborhood. Since the interfaces are not dense, and we assume convex foliation, we may choose $B_x$ small enough so that all rays corresponding to rays staying completely inside $B_x$ do not reach an interface.

The proof of \cite[Theorem 10.2]{SUV2019transmission} (\cite{CHKUElastic} has a slightly different proof) shows $c = \tilde c$ in some neighborhood of $x$ inside $B_x$ which we keep denoting as $B_x$. The proof there is in the elastic case but it carries over to our nonlocal scalar wave setting since only the principal symbol of the operator is relevant for that proof to recover local travel times of geodesics. Now let $(z_1, \zeta_1) \in \p T^*\Omega_q$ with $z_1 \in B_x$ such that the ray $\gamma=\gamma_{z_1,\zeta_1}$ starting at $(z_1,\zeta_1)$ at time $t=0$ remains inside $B_x$ until it hits $\Omega_{\T}$ again at some point $(z_2,\zeta_2) $ at time $t_2$. Let $\Lambda_t$ be the Lagrangian associated to the bicharacteristic flow of time $t$, but restricted to $B_x$. To describe the Lagrangian $\Lambda_{t}$,
denote $\chi_t : T^*\R^n \to T^*\R^n$ the bicharacteristic flow by $t$ units of time; that is, if $\alpha(t)$ is a smooth bicharacteristic with $\alpha(0) = (y, \eta)$, then $\chi_t(y, \eta) = \alpha(t)$.  We can then write $\Lambda_{t} = \{ (\chi_{t}(y,\eta), y, \eta); (y,\eta) \in T^*B_x\}$.

Let $V \in \mathcal E'(\Omega)$, the space of compactly supported distributions in $\Omega$, with $\text{WF}(V) = \R_+(z_1, \zeta_1)$. As shown in \cite[Proof of Theorem 10.2]{SUV2019transmission} using proposition \ref{Prop SC tail}, we may construct an outgoing wave $u, \tilde u$ in $\Omega_q$ so that near $z_1$, $u(T) \equiv V$ and $\tilde u (T) \equiv V$ at some time $T$.
As in \cite{SUV2019transmission}, $\gamma$ does not hit an interface for either operator.
We now consider two cases, depending on whether $x$ is on an interface or not. Without loss of generality, we assume $x$ is at an interface, and the other case follows by an analogous argument.

 We want measurements of the amplitude of $u,\tilde u$ on $\Sigma_+ \supset \Gamma_+$, while our assumptions only allow measurements on $\Gamma_-$ where we have already recovered the operator. We are essentially treating $\Gamma_+$ (and not $\Gamma_-$ as the boundary for $\Omega_\T$). Fortunately, the difference is given by the transmission operator which we can recover via Proposition \ref{Main_th_1}. Note that since we recovered the speeds, $\gamma = \tilde \gamma$, and $\Lambda_t = \tilde \Lambda_t$.

Next, we want to apply propagation of singularities and the foliation condition to conclude that $u \equiv \tilde u$ in $\overline \Omega_{\T}^c$ since $\mL = \tilde \mL$ in that region modulo the self-gravitation term. 
 Since the operator is nonlocal, we will give the full justification. Let $w = u-\tilde u$ so $w$ is $C^\infty$ in $\overline \Omega^c$ by our assumption $\mathbf F = \tilde{\mathbf F}.$
Then
\[
\mL w = -\mL \tilde u = (\tilde \mL - \mL) \tilde u
= -\rho B(\rho - \tilde \rho) \tilde u \text{ in } \overline \Omega^c_\T
\]
where we also used that $\Phi = \tilde \Phi$ in $\overline \Omega_{\T}^c$.
Next, we use microlocality
\[
\text{WF}(B(\rho - \tilde \rho) \tilde u)
\subset \text{WF}((\rho-\tilde \rho)\tilde u)
\subset \overline T^*\Omega_\T
\]
since $\rho=\tilde \rho$ in $ \overline \Omega^c_\T$.
Thus, $B(\rho-\tilde \rho) \tilde u$ is $C^\infty$ over $\overline \Omega^c_\T$ so we conclude that $\mL w \in C^\infty $ over $\overline \Omega^c_\T$.

Since $w \in C^\infty$ in $\Omega^c$, then by propagation of singularities and the convex foliation condition, we may conclude that $w\in C^\infty$ on $\overline \Omega_\T^c$. This is because if there was some $(x,\xi) \in T^*\overline \Omega_\T^c \cap \text{WF}(w)$ whose associated ray is nonglancing, then by convex foliation, either the forward broken bicharacteristic or the backward bicharacteristic through $(x,\xi)$ enters $\Omega^c$ due to the convex foliation assumption. By propagation of singularities, $w$ would have nonempty singular support outside of $\Omega$, leading to a contradiction. By construction, both $u$ and $\tilde u$ are microsupported outside the glancing set so we conclude that $u \equiv \tilde u$ in $\overline \Omega_\T^c$.

 Next, the same proof as \cite[lemma 5.6]{CHKUElastic} implies $M_R = \tilde M_R$, that is, the reflection operators may be recovered from the outside measurement operator due to the convex foliation condition. Near $(T,z_1)$, by construction we have $u|_{\Gamma_+} = M_T u|_{\Gamma_-}$ and $\tilde u|_{\Gamma_+} = \tilde M_T \tilde u|_{\Gamma_-}$.
By Proposition \ref{Main_th_1} and since $u|_{\Gamma_-} \equiv \tilde u|_{\Gamma_-}$, we conclude $u|_{\Gamma_+} \equiv \tilde u|_{\Gamma_+} \equiv V$ near $(T,z_1)$.

By the construction in proposition \ref{Prop SC tail}, $u$ and $\tilde u$ inside $B = [T,T+t_2] \times B_x$ are given by the forward wave propagator applied to $V$ with wavefront set in $\Sigma$ and can be represented by a Fourier integral operator described in section \ref{Sec_geometric_cond}.

Near $(T+t_2,z_2)$, we have $u|_{\Gamma_-} \equiv M_T u_{\Gamma_+}$ and $\tilde u|_{\Gamma_-} \equiv \tilde M_T \tilde u_{\Gamma_+}$. Since $u \equiv \tilde u$ outside $\bar\Omega_q$ and using Proposition \ref{Main_th_1}, we conclude $u|_{\Gamma_+} \equiv \tilde u|_{\Gamma_+}$ near $(T+t_2,z_2)$ as well.
 Denote $\rho_{\Gamma_+}$ as the restriction of $\rho$ to $\Gamma$, from below, near $z_2$. As shown in \cite[Chapter 5]{DuisFIO}, if $\text{WF}(u)$ over $\Gamma_+$ contains no covectors tangential to $\Gamma$, then $\rho_{\Gamma_+}u$ is also the image of an FIO in $\mathcal{I}^0$ (applied to $V$) associated to a canonical graph Lagrangian $\Lambda \subset T^*(\Gamma_+ \times \R_t) \times T^*B_x$. The Lagrangian can be described using $\chi_t$ from before. We compute $\Lambda = \{ t(y,\eta), -|\eta|_P, d\rho_{\Gamma_+}\chi_{t(y,\eta)}(y,\eta), (y, \eta)\}$ where $t(y,\eta)$ is the time the ray from $(y, \eta)$ hits $\Gamma$ for the first time. It will be convenient to just define $\Phi(y,\eta) =  (t(y,\eta), -|\eta|_c, d\rho_{\Gamma_+}\chi_{t(y,\eta)})$. Let us introduce boundary normal coordinates for $\Gamma$ near $z_2$ with local coordinates $(x',x_n)$, where $\Gamma$ is given by $x_n = 0$. As in \cite[equation (5.5) and (5.6)]{BHKUdensityInterior}, we then have near $(t_2,z_2)$
 \begin{equation}
 \rho_{\Gamma_+}u \equiv \int e^{i \phi^+(t,x',\eta)} a_{+}(t,x',\eta) \hat{V}(\eta) \ d\eta
\end{equation}
and likewise
\begin{equation}
\rho_{\Gamma_+}\tilde u \equiv \int e^{i \tilde \phi^+(t,x',\eta)} \tilde a_{+}(t,x',\eta) \hat{V}(\eta) \ d\eta.
\end{equation}

Since we have recovered $c$ already, then $\phi^+ = \tilde \phi^+$ in $B_x$ since they only depend on the principal symbol. Now we may apply the argument in \cite[proof of theorem 1.1]{BHKUdensityInterior} to conclude 
\begin{equation}\label{e: a_k = tilde a_k for each k}
(a_{+})_k = (\tilde a_{+})_k
\end{equation}
 for each $k=0,-1,-2,\dots$ on $\Lambda$ (where these amplitudes may be taken as functions on $\Lambda$ via a diffeomorphsim \cite[Chapter 25]{Hormander_1}) modulo $S^{-\infty}$. Again, the nonlocal term does not affect the argument as long as $a_+$ has a polyhomogeneous expansion, which it does using the construction in section \ref{Sec_geometric_cond}.

\subsection*{Recovery of the parameters}
We will now use \eqref{e: a_k = tilde a_k for each k} and omit the $+$ subscript to recover the coefficients. First let us observe that the projection $\underline{\pi}$ to the base space $\Omega$ of $\zeta_{\pm}$ are geodesics $\gamma_{\pm}$ in $(\Omega,g_c)$.
Let us select $\zeta_{+}$ such that $\underline{\pi}\zeta_{+}=\gamma_{+}$ does not intersect any other interface between entering and exiting $\Omega_{\T}$.

Let us set $\underline{s}, \overline{s} \in \R$ to be the entering and exiting points of $\zeta_{+}(\cdot)$ in $\Omega_{\T}$. That is, $\zeta_{+} \cap \Omega_{\T}$ is the curve $\{\zeta_{+}(s) : \underline{s} \leq s \leq \overline{s}\}$. To ease notation, we write $f(s)$ in place of $f(t(s),x(s),\xi(s))$ for functions restricted to the bicharacteristic, and also $\xi = \nabla_x \varphi$. Also $x(\underline s) \notin \Omega^o$.
From \eqref{transport_eq_J}, we obtain

\begin{equation}\label{term_a_0}
\quad (\alpha)_0(s) = (\alpha)_0(\underline{s})\sqrt{\frac{\rho(\underline s)c(\underline s)} {\rho(s)c(s)}} \exp\left[\int_{\underline{s}}^{s}-\frac 1 2 \nabla_x\cdot \frac{\xi}{|\xi|}\ dr\right] = \frac{H}{\sqrt \rho},
\end{equation}
where $H = H(s)$ depends on $c$, but not on $\rho(x)$ for $x \in \Omega^o$. It is convenient to denote
\[
g:= \frac{\sqrt \rho}{H} = \frac{1}{(\alpha)_0}
= \frac{\sqrt{\rho(s)c(s)}}
{(\alpha)_0(\underline{s})\sqrt{\rho(\underline s)c(\underline s)}}  \exp\left[\int_{\underline{s}}^{s}\frac 1 2 \nabla_x\cdot \frac{\xi}{|\xi|}\ dr\right].
\]

Next, we consider the compatibility condition \eqref{transport_eq_Interior} for $J = -1$, which reduces to the transport equation
\[
\frac{d}{ds} (\alpha)_{-1}
+ \left[ \frac{d}{ds} \log \sqrt{\rho c} +
\frac 1 2 \nabla_x\cdot \frac{\xi}{|\xi|}\,
 \right] (\alpha)_{-1} = G
\]
where 
\[
G = \frac{-1}{2i\rho c^2|\xi|}
[N B_p h_{-1} + N C_p(\alpha)_0 N + N\cdot (c)_0 N]
\]
To solve this explicitly, first observe that using \eqref{transport_eq_Interior},
\begin{align*}
\frac {d} {ds}
\frac{(\alpha)_{-1}}{(\alpha)_0}
&= \frac{\frac{d}{ds}(\alpha)_{-1}}{(\alpha)_0}
- \frac{(\alpha)_{-1}}{(\alpha)^2_0} \frac{d}{ds}(\alpha)_0
\\
&=\frac{\frac{d}{ds}(\alpha)_{-1}+
\frac{d}{ds} \log \sqrt{\rho c} (\alpha)_{-1} + \frac 1 2 (\nabla_x \cdot N) 
(\alpha)_{-1}}{(\alpha)_0}
\\
&= \frac{G}{(\alpha)_0}
\end{align*}
so we obtain
\[
\frac{(\alpha)_{-1}(s)}{(\alpha)_0(s)} 
=  \int_{\underline s}^s \frac{G}{(\alpha)_0} \ dr
+ (\alpha)_{-1}(\underline s)/(\alpha)_{0}(\underline s)
\]
\[
g(s) (\alpha)_{-1}(s)
= \int_{\underline s}^s g(r)G(r) \ dr
+ (\alpha)_{-1}(\underline s) g(\underline s)
\]

We compute $G$ via the following lemma.
\begin{Lemma} \label{l: calculating G} Assuming $G,g,H$ as above, we have
\begin{align}\label{e: integral of tensor A}
g(\alpha)_{-1} &=\int gG + C \nonumber \\
&= \int \frac{\sqrt \rho}{H}\cdot \frac{-1}{2i\rho c^2|\xi|} \cdot(-H\sqrt \rho)[N^t \cdot A(x) \cdot N] + C
\end{align}
where the rank two tensor $A(x)$ depends only on $x$ and is given by
\begin{multline}\label{e: A(x) formula in lemma}
A(x) = c^2 \nabla_x^2 \log \sqrt \rho 
+ 2\nabla_x \log \sqrt \rho \otimes \nabla_x c^2
\\
- I \Bigg[ c^2 \Delta(\log \sqrt \rho)] 
-c^2|\nabla_x \log \sqrt \rho |^2
+ \nabla_x \log \sqrt \rho \cdot \nabla_x c^2 \Bigg]
\\
- \nabla_x \log \sqrt \rho \otimes \dPhi-   \frac{\nabla_x H}{H}  \otimes \dPhi + (\nabla_x \otimes N)\dPhi \otimes N\\
- (\nabla_x \cdot N)  \dPhi \otimes N
+ \nabla_x^2 \Phi
\\
+I\Bigg[c^2 \frac{|\xi|}{H}\nabla_x \frac{H}{c^2|\xi|} \cdot \dPhi- \nabla_x \log \sqrt \rho \cdot \dPhi + \Delta_x \Phi - c^{-2}|\dPhi|^2 - \rho \Bigg]
\\
+(\text{ terms independent of $\rho$ and $\Phi$})
\end{multline}
\end{Lemma}

Since it is a lengthy calculation, we prove the lemma in appendix \ref{a: proving formula for A(x)}.
Similar to \cite{RachDensity}, these line integrals \eqref{e: integral of tensor A}, when taken over bicharacteristics starting at $(\underline x, \underline \xi) \in \p_+ S^*\Omega_{\mathbf q}$, make up a local ray transform of the tensor field $A$ over the family of geodesics associated to the metric $c^{-2} dx^2$. Thus, since we can uniquely recover $(\alpha)_{-1}$ at $\p \Omega_{\mathbf q}$ using proposition \ref{Prop SC tail}, then we recover the local ray transform of $A$ on a dense subset of rays. In other words, there exists an open neighborhood $\mathcal O$ of $x$ within $\Omega_{\mathbf q}$, such that for a dense set of points $(\underline x, \overline x) \in (\p \Omega_{\mathbf q} \cap \mathcal O) \times (\p \Omega_{\mathbf q} \cap \mathcal O)$, we can conclude that
\begin{equation}\label{e: int of B = 0}
\int_ \gamma N^t \cdot B(x) \cdot N \ ds = 0
\end{equation}
where $\gamma$ is the geodesic segment between $\underline x$ and $\overline x$
and 
\begin{multline}\label{e: B(x) formula}
B(x) = \frac{A}{c} - \frac{\tilde A}{c}
\\
= \kappa \Bigg[ \nabla_x \log \sqrt \rho  \otimes \nabla_x \log \sqrt{ \rho}
-\nabla_x \log \sqrt{ \tilde\rho}  \otimes \nabla_x \log \sqrt{\tilde \rho}
\Bigg] - \alpha I
\\
+ 2\left[\nabla_x\log \sqrt{\frac{\rho}{\tilde \rho}} \cdot \nabla_x c \right] I
+ 4 \sigma_{ij}\nabla_x\log \sqrt{\frac{\rho}{\tilde \rho}} \otimes \nabla_x c + 2c^2 \left( \nabla^2_x \log \sqrt{\frac{\rho}{\tilde \rho}}  \right) 
\\
+ 2c \sigma_{ij}\left(\dPhi-\nabla_x \tilde \Phi\right) \otimes V
+ c\nabla_x^2(\Phi - \tilde \Phi)
\end{multline}
where $\sigma_{ij}$ denotes the algebraic symmetrization of a 2-tensor, and
\begin{multline}
\alpha = c \Delta \log \sqrt{\frac{\rho}{\tilde \rho}} - c \left[\left|\nabla_x \log \sqrt \rho \right|^2 - |\nabla_x \log \sqrt {\tilde \rho} |^2  \right]-c(\rho- \tilde \rho)
\\
+c(\Delta_x \Phi - \Delta_x \tilde \Phi)
-c( \nabla_x \log \sqrt \rho \cdot \dPhi -\nabla_x \log \sqrt{\tilde \rho} \cdot \nabla_x \tilde \Phi)
\\ + c\frac{|\xi|}{H}\nabla_x \frac{H}{c^2|\xi|} \cdot ( \dPhi - \nabla_x \tilde \Phi) - c(|\dPhi|^2- |\nabla_x \tilde \Phi|^2), 
\end{multline}
\[
V = - \nabla_x \log \sqrt {\frac{\rho}{\tilde \rho}} - \frac{\nabla_x H}{H}
+ (\nabla_x \otimes N)^t N - (\nabla_x \cdot N) N.
\]

Next, the Saint-Venant operator $W$ is defined on the space of smooth, symmetrix (here, rank-2) tensor fields mapping into the smooth, symmetric (here, rank-4) tensor fields by
\[
(WB)_{i_1i_2j_1j_2} = B_{j_1j_2,i_1i_2} + B_{j_1j_2,i_1i_2}-B_{i_1j_2,i_2j_1} - B_{i_2j_1,i_1j_2}
\]
and $B_{,kl}$ represents differentiation by $x_k$ and $x_l$.
Then $W(\nabla_x^s v) = 0$ where $\nabla_x^s$ is the symmetric gradient. In particular $W(\nabla_x^2 f) = 0$ for any scalar function $f$. 

Using \eqref{e: int of B = 0}, we can apply the argument in \cite{Rach00} and \cite{B_density} that use the local tensor rigidity result of \cite{SUVlocaltensor}, to conclude that there is an open neighborhood $\mathcal O_x$ of $x$ within $\Omega_{\mathbf q}$, such that within this neighborhood, $B = dv$, where $v$ is a smooth 1-tensor that vanishes on $\p \Omega_{\mathbf q} \cap \mathcal O_x$.
 We note that $v$ involves derivatives of $\rho_i(x), \Phi_i(x)$, $i=1,2$ of order at most $1$ since $B$ involves derivatives of $\rho_i, \Phi_i$ of order at most $2$. We can write $dv = \nabla_x \circledS v + R(v),$ where $R(v)$ depends on $v$ and derivatives of $c$. It follows from a short computation that
\[
W(B) = W(dv) = W(\nabla_x \circledS v) + W R(v) = W R(v)
\]
involves derivatives of $\rho_i(x), \Phi_i(x)$ of order at most $3$. We write $T_l u$ for the sum of terms of $u$ that involve the $l$th-order derivatives of $\rho_i, \Phi_i$, $i=1,2$. Then
\[
T_4 W(B) = 0.
\]
From this fact, now we derive a partial differential equation whose solution, involves the differences of $\rho,\tilde \rho$ and $\Phi,\tilde \Phi$.

From the computation in \cite[section 4]{RachDensity}, that terms of the form $W(\nabla_x f \otimes \nabla_x f)(X,X,Y,Y)$ only involve two derivatives of $f$ and will not appear when we compute $T_4 W(B)[e_i,e_i,e_j,e_j]$ where $e_k$ is the standard unit vector in $\R^3$ with $1$ in the $k$th component.
Denote $\beta = \log \sqrt \rho$ and $\tilde \beta = \log \sqrt {\tilde \rho}$.
We sum over $i,j = 1,2,3$ to conclude, given \eqref{e: B(x) formula}, 
\begin{align*}
0  &= \sum_{ij} T_4 W(B)[e_i,e_i,e_j,e_j]
\\
&= T_4 \sum_{ij}
\Big[
\kappa W(\nabla_x \beta \otimes \nabla_x \beta) - \kappa W(\nabla_x \tilde \beta \otimes \nabla_x \tilde \beta) - W(\alpha I)
\Big](e_i,e_i,e_j,e_j)
\end{align*}
\begin{multline}\label{e: elliptic_eq_1}
=\left[\
\D^2_x\log\sqrt{\frac{\rho}{\tilde{\rho}}} +\Delta_x(|\nabla_x\log\sqrt{\rho}|^2 - |\nabla_x\log\sqrt{\tilde{\rho}}|^2)\right]
 \\
+ \D^2_x (\Phi -  \tilde \Phi) -\Delta_x(|\dPhi|^2 - |\nabla_x \tilde \Phi|^2)
\end{multline}

We can extend all the parameters smoothly outside $\Sigma_{\T}$ to obtain a larger open set $\tilde{U}\supset \mathcal O_{x}$ where $\rho=\tilde{\rho}$ and $\Phi = \tilde \Phi$ outside $\overline \Omega_{\T}$ in $\tilde{U}$.

Next, note that $\Delta_x \Phi = k_0 \rho$ for some constant $k_0$ independent of the coefficients.
Set $\beta_{\pm}:=\log\sqrt{\rho}\pm\log\sqrt{\tilde{\rho}}$ and rewrite \eqref{e: elliptic_eq_1} as
\begin{multline}\label{e: elliptic_eq_2}
 \D^2_x\beta_{-} + \Delta_x \nabla_x\beta_{+} \cdot \nabla_x\beta_{-} 
 + k_0 \Delta_x (\rho - \tilde \rho)
  -\Delta_x(|\dPhi|^2 - |\nabla_x \tilde \Phi|^2)= 0, \quad \mbox{in }\tilde{U},
\end{multline}
and $\beta_{-}, (\Phi-\tilde \Phi)=0$ in $\tilde{U}\setminus U$. As a single equation for $\rho$, is nonlocal due to the $\dPhi$ terms and we cannot apply unique continuation results. Instead, let us denote \[Y := \Phi - \tilde \Phi\] so $Y$ satisfies 
\[
\Delta^2 Y = k_0 \Delta(\rho - \tilde \rho).
\]
By writing $\rho = e^{2\log \sqrt \rho}$, we obtain a linear elliptic system of equations for $(\beta_-, Y)$:

\begin{align}
\D^2_x\beta_{-} + &\Delta_x \nabla_x\beta_{+} \cdot \nabla_x\beta_{-} 
 + k_0 \Delta_x g(x)(e^{\beta_-} - 1)
  -\Delta_x h(x)\cdot \nabla_x Y = 0 \nonumber
  \\
  \Delta^2 Y &= k_0 \Delta g(x) (e^{\beta_-} - 1), \quad \mbox{in }\tilde{U},
\end{align}
where $h(x) := \nabla_x \Phi + \nabla_x \tilde \Phi $ and $g(x):= \tilde \rho. $

Due to the assumptions \eqref{e: lispchitz rho assumption}, this system of equations implies
\begin{align} \label{e: f_1 and f_2 Lipschitz inequality}
|\Delta_x^2 \beta_-| &\lesssim f_1(x,\nabla_x (\beta_-, Y), \nabla_x^2 (\beta_-, Y), \nabla_x^3 (\beta_-, Y)) \nonumber
\\
|\Delta^2 Y| & \lesssim f_2(x,\beta_-, \nabla_x \beta_-, \nabla_x^2 \beta_-)
\end{align}
where $f_1(x,p,q,r)$ is locally Lipschitz in the variables $p \in \R^3 \times \R^4$, $q \in \R^9  \times \R^4$, $r \in \R^{27} \times \R^4$, and analogously for $f_2$. The nonlinearity involving the exponential poses no difficulties since we may restrict to an open subset where $|\beta_-|, |Y| \leq 1$, to obtain the Lispchitz estimate \eqref{e: f_1 and f_2 Lipschitz inequality}. This is because for the unique continuation argument, one starts with the assumption that $\beta_-, Y$ are $0$ in some small open set, and we only need to prove that this open joint zero set may be extended. We will now apply unique continuation to this joint elliptic equation. By proposition \ref{prop: elliptic ucp} (elliptic unique continuation), there exists a neighborhood of $x$ such that $Y, \beta_- = 0$ since $Y,\beta_- = 0$ in $\Omega_\mathbf q^c$.

\end{proof}
\subsubsection{Elliptic unique continuation} To obtain an elliptic unique continuation theorem, need a Carleman estimate from \cite{UWElasticCarleman}. Set
$A = \sum_{ij}a_{jk}(x) \p^2_{x_j x_k}$ with $a_{jk}(x) = a_{kj}(x), a_{ij} \in W_{loc}^{1,\infty}(\R^n)$ for $1 \leq i,j \leq n,$ and assume $A$ is elliptic.

Let $r_0 < 1$ and $U_{r_0} = \{ u \in C^\infty_0(\R^n \setminus \{0\}): \text{supp}(u) \subset B_{r_0}\}$ where $B_{r_0}$ is the ball centered at the origin with radius $r_0$. Denote $r = |x|$, $\psi = \exp(r^{-\beta})$ and $ s = s_0 + \tilde c \beta$ with $s_0, \tilde c \in \R.$

\begin{Proposition}
    There exists a positive constant $\beta_0$ such that for all $\beta \geq \beta_0$ and $u \in U_{r_0}$ with $r_0$ small enough, we have
    \begin{equation}\label{e: carleman 2nd order}
    \beta^2 \int r^{-s-\beta-1} \psi^2 (|\nabla u|^2 + |u|^2) dx
    \leq c\int r^{-s}\psi^2 |Au|^2 dx.
    \end{equation}
    The constant $c$ is independent of $\beta$ and $u$
\end{Proposition}
As shown in \cite{Protter}, such an estimate can be iterated to show that $u, \nabla u, \nabla^2 u, \nabla ^3 u$ can be estimated via $\Delta^2 u$. Thus, we can obtain
\begin{equation}\label{e: carleman 4th order}
  \beta^4 \int r^{-s-6\beta-8} \psi^2 ( |\nabla^3 u|^2+ |\nabla^2 u|^2+|\nabla u|^2 + |u|^2) dx
    \leq c\int r^{-s}\psi^2 |\Delta^2 u|^2 dx.  
\end{equation}

The goal is to prove the following.
\begin{Proposition}\label{prop: elliptic ucp}
Suppose $(\beta_-, Y)$ satisfies the inequality \eqref{e: f_1 and f_2 Lipschitz inequality} on an open set $U$ and $(\beta_-, Y) = 0$ in a non-empty open subset of $U$. Then $(\beta_-, Y) = 0$ in $\Omega$.
\end{Proposition}
\begin{proof}
The proof follows a similar structure to the proof in \cite{UWElasticCarleman}. Suppose $(\beta_-, Y)$ vanishes in a neighborhood of $x_0 \in U$. Without loss of generality, we assume $x_0 = 0$. We set $\tilde r = \min{(r_0, 1/2, \text{dist}(0,\p U))}$. Now let $\chi \in C^\infty_0(\R^3)$ be a cut-off function satisfying $0 \leq \chi \leq 1$, $\chi|_{B_{r/2}} = 1$ and $\text{supp}(\chi) \subset B_{\tilde r}$. Denote $w_1 = \chi \beta_-$ and $w_2 = \chi Y$. Set $e(w) = |\nabla^3 w|^3 + |\nabla^2 w|^2+ |\nabla w|^2 + |w|^2$. Next observe that using \eqref{e: f_1 and f_2 Lipschitz inequality}, we have
\begin{align*}
|\Delta^2 w_1| = |\chi \Delta^2 \beta_- + [\chi, \Delta^2] \beta_-|
&\lesssim
|\chi f_1(x,\nabla_x (\beta_-, Y), \nabla_x^2 (\beta_-, Y), \nabla_x^3 (\beta_-, Y))| + 
|[\chi, \Delta^2] \beta_-|
\\
&\lesssim
| \chi (e(\beta_-)+ e(Y))^{1/2}|
+ |[\chi, \Delta^2] \beta_-|
\\
&\lesssim
(e(w_1)+ e(w_2))^{1/2} + g_1
\end{align*}

where $g_1$ only contains terms that have at least on derivatice of $\chi$ so it is supported in $B_{\tilde r} \setminus B_{\tilde r/2}$.
Likewise, we have
\[
|\Delta^2 w_2|
\lesssim (e(w_1) + e(w_2))^{1/2} + g_2
\]
with $g_2$ supported in $B_{\tilde r} \setminus B_{\tilde r/2}$.

By combining the above estimate with \eqref{e: carleman 4th order}, we obtain
\[
\beta^2 \int r^{-s-6\beta-8}\psi^2 e(w_j) \leq c \int r^{-s} \psi^2 (e(w_1)+e(w_2)) \ dx
+ \int_{B_{\tilde r} \setminus B_{\tilde r/2}} g_j
\]
so combining the estimate for each $j$ gives
\[
\beta^2 \int r^{-s-6\beta-8}\psi^2 (e(w_1) +e(w_2))\leq c \int r^{-s} \psi^2 (e(w_1)+e(w_2)) \ dx
+ \int_{B_{\tilde r} \setminus B_{\tilde r/2}} r^{-s}\psi^2(g_1 + g_2) \ dx.
\]
Since $r < 1$, the $r^{-s} < r^{-2-6\beta - 8}$ so with a large $\beta$, we may absorb the first term on the right hand side into the left hand side:
\[
\beta^2 \int r^{-s-6\beta-8}\psi^2 (e(w_1) +e(w_2))\leq c  \int_{B_{\tilde r} \setminus B_{\tilde r/2}}  r^{-s-6\beta-8}\psi^2 (g_1 + g_2) \ dx.
\]
We may choose $s$ so that $-s-6\beta-8 = -\beta$ and obtain the estimate
\[
\beta^2 \int_{B_{\tilde r}} r^{-\beta}\psi^2 (e(w_1) +e(w_2))\leq c  \int_{B_{\tilde r} \setminus B_{\tilde r/2}}  r^{-\beta}\psi^2 (g_1 + g_2) \ dx.
\]
Since $\tilde \psi(r) := r^{-\beta} \psi^2$ is a strictly decreasing function, we obtain
\[
\beta^2 \tilde \psi(\tilde r/2)\int_{B_{\tilde r/2}}  ((e(w_1) +e(w_2))\leq c  \tilde \psi(\tilde r) \int_{B_{\tilde r} \setminus B_{\tilde r/2}} (g_1 + g_2) \ dx
\]
so we finally get
\[
\int_{B_{\tilde r/2}}  ((e(w_1) +e(w_2))\leq \frac{c}{\beta^2}  \int_{B_{\tilde r} \setminus B_{\tilde r/2}} (g_1 + g_2) \ dx.
\]
By letting $\beta \to \infty$, this shows $\beta_-, Y = 0$ in $B_{\tilde r/2}$. By iteration, even if done infinitely many times, a standard argument proves the proposition.

\end{proof}

\section{Discussion} \label{s: discussion}
This manuscript presents an inverse problem for an acoustic-gravitational system that include a nonlocal operator contribution and piecewise smooth coefficients, but whose principal symbol is the same as a scalar wave equation. We anticipate that an analogous proof would apply to the full isotropic elastic operator with two wave speeds, although it would involve more tedious computations, which we do not pursue in this work.

We also anticipate that our proof will hold for an inverse problem concerning a ``truncated'' gas giant described in \cite{gasGiants2024}.
On gas giant planets, the speed of sound is isotropic and goes to zero at the surface. Geometrically, this would correspond to a Riemannian manifold whose metric tensor has a conformal blow-up near the boundary. If $M$ represents such a planet, with boundary defining function $x$, then a small truncation of the planet is represented by $\Omega = \{ p \in M: x(p) \geq \epsilon\}$ for a small $\epsilon$ so that $\rho, c$ are strictly positive on $\Omega$. We assume that for $\epsilon_0>0$, $\epsilon \geq \epsilon_0$, then $c,\rho,\Phi$ may be measured and known in $M \setminus \Omega$. 
\section{Declarations}
\subsection*{Funding} M.V.d.H. gratefully acknowledges support from the Simons Foundation under the MATH + X program, the National Science Foundation under grant
DMS-1815143, and the corporate members of the Geo-Mathematical Imaging Group at Rice University. S.B. was partly supported by Project no.: 16305018 of the Hong Kong Research Grant Council.

\subsection*{Conflict of interest/Competing interests}
{\bf Financial interests:} The authors declare they have no financial interests.
\\

\noindent {\bf Non-financial interests:} The authors declare they have no non-financial interests.

\appendix

\section{Proof of theorem \ref{Main_th_1}: Interface determination}\label{Sec_interface-determination}
In this section, we prove Theorem \ref{Main_th_1}. 
Let $\Gamma \subset \R^n$ be a bounded, closed, smooth hypersurface and $\Oo$ be a neighbourhood of $\Gamma$ in $\R^n$. Without loss of generality, up to a diffeomorphism, we can take $\Gamma \subset \{x_3=0\}$ and consider $\Oo_{\pm}\subset\{\pm x_3 > 0\}$ to be small neighbourhood of $\Gamma$ on both sides of it.
We call $x_3>0$ to be `above' while $x_3<0$ to be `below' the interface.
We write $x=(x',x_3)$ as local coordinates in $\Oo$ and $\nu= (0,0,1)^t$ as a unit normal vector.
In $\Oo$, we consider the boundary value problem
\begin{equation}\label{BVP}
\begin{aligned}
\mL u =& 0 \quad &&\mbox{in }\R_+\times\Oo,\\
u(t,x') =& h(t,x') \quad &&\mbox{on }\R_+\times\Gamma,\\
u(0,x) = 0 \quad &\p_tu(0,x)=0 \quad&&\mbox{in }\Oo.
\end{aligned}
\end{equation}
 In this section, let use first assume $B(x,D) = 0$ for a clearer exposition, and we will show how to modify the arguments afterwards with $B(x,D)$. 

 We will first recall several facts. Let $\eta \in \R\times\R^2$ be the dual coordinate of $(t,x') \in \R\times\R^2$.
The phase function $\varphi$ is given as a solution of the Eikonal equation
\begin{equation}\label{Eikonal_eq}
|\p_t \varphi|^2 = c^2(x)|\nabla_x\varphi|^2,
\quad \mbox{with the condition}\quad
\varphi|_{x_3=0} = -t\tau + x'\cdot\eta',
\end{equation}
where $c(x):=\sqrt{\frac{\kappa}{\rho}}$ is the piece-wise smooth wave speed in $\Oo$. From the Eikonal equation \eqref{Eikonal_eq} we clearly observe that $\p_t\varphi = -\tau$, $\nabla_{x'}\varphi = \eta'$ at $\R_+\times\Gamma$ and \[\p_{x_3}\varphi = \sqrt{c^{-2}|\p_t\varphi|^2 - |\nabla_{x'}\varphi|^2} = \sqrt{c^{-2}\tau^2 - |\eta'|^2},\qquad \mbox{on }\R_{+}\times\Gamma.\]
For notational convenience let us define $\xi_3 = \sqrt{c^{-2}\tau^2 - |\eta'|^2}$ and note that $\p_{x_3}\varphi = \xi_3$ at $\Gamma$. We also denote $\xi = \nabla_x \varphi$.

 We work with the parametrix
\begin{equation}\label{wave_parametrix}
u(t,x,\tau,\eta') = e^{i\varphi(t,x,\tau,\eta')}(\alpha N + h_{-1})
\end{equation}
introduced earlier.

\subsection{Wave transmission conditions}\label{Sec_reflection_transmission}

Let $u_I$ be the parametrix of an incoming wave-field travelling through $\mathcal{O}_{-}$ \footnote{See Section \ref{Sec_intro_interface_determination} for the definition of $\mathcal{O}_{\pm}$ and $\Gamma_{\pm}$.} towards the interface $\Gamma$. After hitting $\Gamma$, $u_I$ splits into two parts $u_R$ and $u_T$ the reflected and the transmitted wave-fields respectively. The reflected wave-field $u_R$ travels through $\mathcal{O}_{-}$ and $u_T$ travels through $\mathcal{O}_{+}$.
We add the suffix $\pm$ to denote restriction on either sides of $\Gamma$, for instance $c_{\pm}$ denote $c$ restricted on $\overline{\mathcal{O}}_{\pm}$. We write $u_{\bullet}$ to denote $u_I$, $u_R$ or $u_T$, where there is no ambiguity.
Sometimes, for notational convenience, we write $f_{\bullet}$ for a function $f$ in $\mathcal{O}$ to denote $f_{I}=f_R = f_{-}$ and $f_T=f_{+}$.

We denote the Neumann data $\mathcal{N}_{\bullet}$ at $\Gamma_{\pm}$ as
\[ \mathcal{N}_{\bullet} u_{\bullet} := \kappa_\bullet \nu \nabla_x \cdot u_\bullet
- \nu \rho \langle u_\bullet , \dPhi \rangle
+ \rho \dPhi \langle u_\bullet , \nu\rangle 
|_{\Gamma_{\pm}}.
\]
Since $\Gamma\subset \{x_3=0\}$ and $u_{\bullet}=e^{i\varphi_{\bullet}}a_{\bullet}$ a straightforward calculation implies
\[ \mathcal{N}_{\bullet}u_{\bullet}
= e^{i\varphi_{\bullet}} \left[ i\nu \kappa_{\bullet} \alpha_\bullet \langle \nabla_x \varphi_{\bullet}, N_\bullet \rangle 
+\kappa_\bullet \nu  (\nabla_x \cdot  a)
- \nu \rho \langle a_\bullet, \dPhi \rangle
+ \rho \dPhi \langle a_\bullet, \nu \rangle 
\right].
\]

Define \[\xi_{\bullet,3} = \sqrt{c^{-2}|\p_{t}\varphi_{\bullet}|^2 - |\nabla_{x'}\varphi_{\bullet}|^2}\] and note that $\xi_{R,3} = -\xi_{I,3}$. In general the transmission conditions become
\begin{multline}\label{transmission_cond_J}
\hat \xi_{I,3} \alpha_I - \hat \xi_{I,3} \alpha_R
= \hat \xi_{T,3}\alpha_T
+ \langle \nu, h_{T,-1} - h_{I,-1}-h_{R,-1} \rangle 
\\
\nu \kappa_- i|\xi_I|\alpha_I + \nu \kappa_- i |\xi_R| \alpha_R
= \nu \kappa_+ |\xi_{T}|\alpha_T
+ \kappa_+ \nabla_x \cdot (\alpha_T N_T) - \kappa_-\nabla_x \cdot( \alpha_I N_I + \alpha_R N_R)
\\
+ \kappa_+ \nabla_x \cdot h_{T,-1} 
- \kappa_- \nabla_x \cdot (h_{I,-1} + h_{R,-1})
\\
- \nu \langle \rho_+ u_T
- \rho_- u_I - \rho_- u_R, \dPhi \rangle
+\dPhi \langle u_T-u_I-u_R, \nu \rangle 
\end{multline}
 In matrix form, when separating the equations with the same order we have
\begin{multline}\label{e: Jth transmission cond}
M \col{(\alpha_R)_J \\ (\alpha_T)_J}
:=
\bmat{\hat \xi_{I,3} & \hat \xi_{T,3} \\
i\kappa_-|\xi_I| & -i\kappa_+ |\xi_T|}\col{(\alpha_R)_J \\ (\alpha_T)_J}
= \col{\hat\xi_{I,3} (\alpha_I)_J - 
\langle \nu, (h_{T})_J - (h_{I})_J-(h_{R})_J \rangle 
\\
-i\kappa_+ |\xi_I|(\alpha_I)_J + V_J} 
\end{multline}
where we already showed $M$ is invertible away from the critical angle and
\begin{multline}
V_J 
=  \kappa_+ \nabla_x \cdot ((\alpha_T)_{J+1} N_T) - \kappa_-\nabla_x \cdot( (\alpha_I)_{J+1} N_I + (\alpha_R)_{J+1} N_R)
\\
+ \kappa_+ \nabla_x \cdot (h_{T})_{J+1} 
- \kappa_- \nabla_x \cdot ((h_I)_{J+1} + (h_{R})_{J+1})
\\
- \nu \langle \rho_+ (u_T)_{J+1}
- \rho_- (u_I)_{J+1} - \rho_- (u_R)_{J+1}, \dPhi \rangle
\\
+\dPhi \langle \rho_+(u_T)_{J+1}-\rho_-(u_I)_{J+1}-\rho_- (u_R)_{J+1}, \nu \rangle, 
\end{multline}

Since we must have $u_I\restriction_\Gamma = h$ for a fixed $h$ which is the restriction of an incoming wave at $\Gamma_-$, this imposes the boundary conditions of $(a_{I})_J$, $J=0,-1,\dots$. We get
\begin{equation}\label{e: bdy conditions of a_I}
(\alpha_I)_0 = 1 \text{ and }
(\alpha_I)_J = 0 \text{ on } \R_t \times \Gamma.
\end{equation}

\subsection{Recovering the parameters}
In this section we solve the equations \eqref{e: Jth transmission cond}, for $(a_R)_J$ and obtain the full symbol of $M_R$. One can do the same for $M_T$ as well and get that $M_R$, $M_T$ are $\Psi$DOs of order $0$ at $\left(\R_{+}\times\Gamma_{\pm}\right)$.
In order to prove Theorem \ref{Main_th_1} we assume that there are two sets of parameters $(\rho,\kappa)$ and $(\tilde{\rho},\tilde{\kappa})$ smooth in $\mathcal{O}\setminus\Gamma$. We also assume the notational convention of denoting $\tilde{f}$ to be the corresponding quantity for the parameters $(\tilde{\rho},\tilde{\kappa})$ when $f$ is a quantity corresponding to the parameters $(\rho,\kappa)$.

Let us make the assumption that $(a_R)_J = (\tilde{a}_R)_J$, for each $J=0,-1,-2,\dots$.
Our strategy of recovering the coefficients at the interface is as follows.
We use the transmission conditions \eqref{e: Jth transmission cond} to obtain relations between $(a_T)_J$, $(\tilde{a}_T)_J$ and their normal ($\p_{x_3}$) derivatives at $\R_+\times\Gamma_{+}$.
Then we calculate the transport equations for $(a_T)_J$ derived from \eqref{transport_eq_Interior} and relate $\p_{x_3}(a_T)_J$ with higher order derivatives $\p_{x_3}^{|J|+1}(a_T)_0$ and from their we obtain relations between $\left(\p_{x_3}^{|J|+1}c;\p_{x_3}^{|J|+1}\rho\right)$ and $\left(\p_{x_3}^{|J|+1}\tilde{\rho};\p_{x_3}^{|J|+1}\tilde{c}\right)$.

Observe that since $\xi_{I,3} = -\xi_{R,3}$ and \eqref{e: princ symbol M_R and M_T},

\begin{equation}\label{a_0}
(\alpha_R)_0
= \frac{\xi_{I,3}\rho_+ - \xi_{T,3}\rho_-}{
\xi_{I,3}\rho_+ + \xi_{T,3}\rho_-
}(\alpha_I)_0
\qquad (\alpha_T)_0 
= \frac{2\kappa_- \xi_{I,3}}
{\kappa_+\frac{c_-}{c_+}\xi_{I,3} + \kappa_- \frac{c_+}{c_-}\xi_{T,3}} (\alpha_I)_0, \qquad \mbox{at }\R_+\times\Gamma.
\end{equation}
We state our first interface recovery result for recovering $\kappa,\rho$ at $\Gamma$. We have the following result, analogous to \cite[Lemma 2.1]{BHKU_1}, since $\Phi$ does not appear in \eqref{a_0}. 
\begin{Proposition}\label{Prop_recovery_0}
	Let us assume that $\rho=\tilde{\rho}$, $\kappa=\tilde{\kappa}$, $(a_I)_0 = (\tilde{a}_I)_0$ in $\R_+\times\overline{\mathcal{O}}_{-}$.
	If $(a_R)_0 = (\tilde{a}_R)_0$ at $\R_+\times\Gamma$ then $\rho_+ = \tilde{\rho}_+$, $\kappa_{+} = \tilde{\kappa}_{+}$ and $(a_T)_0 = (\tilde{a}_T)_0$ at $\R_{+}\times\Gamma_{+}$.
\end{Proposition}

\begin{proof}
Assume $(a_R)_0 = (\tilde{a}_R)_0$ at $\R_+\times\Gamma_{-}$ and $(\kappa_-,\rho_-)|_{\Gamma_-} = (\kappa_+,\rho_+)|_{\Gamma_-}$.
From \eqref{a_0} we get
\begin{equation}\label{key_1}
\frac{\rho_{-}\left(\xi_{I,3}/\xi_{T,3}\right) - \rho_{+}}{\rho_{-}\left(\xi_{I,3}/\xi_{T,3}\right) + \rho_{+}}
= \frac{\rho_{-}\left(\xi_{I,3}/\tilde{\xi}_{T,3}\right) - \tilde{\rho}_{+}}{\rho_{-}\left(\xi_{I,3}/\tilde{\xi}_{T,3}\right) + \tilde{\rho}_{+}}, \qquad \mbox{at }\R_+\times\Gamma.
\end{equation}
Let us write $f=\xi_{I,3}/\xi_{T,3}$ and subsequently $\tilde{f}=\xi_{I,3}/\tilde{\xi}_{T,3}$.
From \eqref{key_1} we get
\begin{equation}\label{key_2}
\rho_-\tilde{\rho}_{+}f = \rho_-\rho_{+}\tilde{f}
\quad\Leftrightarrow\quad
\frac{\rho_{+}}{\tilde{\rho}_{+}} = \frac{\tilde{f}}{f}.
\end{equation}
Since $\frac{\tilde{f}}{f}$ is a function of $(x',\frac{|\eta'|}{\tau})$ and
$\rho_{+}$, $\rho{\mu}_{+}$ is a function of $x$ only, therefore, choosing a different value $f_1$ of $f$ by varying $\frac{\eta'}{\tau}$ we get
\[
\frac{\tilde{f}}{f} = \frac{\tilde{f}_1}{f_1}
\quad\Leftrightarrow\quad
\left(\tilde{c}^{-2}_{+} - c^{-2}_{+}\right)(b^2-b_1^2) = 0,
\]
where $b$, $b_1$ are different values of $\frac{|\xi'|}{\tau}$. Therefore, choosing $b \neq \pm b_1$, we have $c_+ = \tilde{c}_+$ on $\Gamma_+$ and thus, from \eqref{key_2} we have $\rho_+ = \tilde{\rho}_+$ on $\Gamma_{+}$.
\end{proof}

Thus we recover $(a_\bullet)_0$ and $(\rho,\mu)$ on $\R_+\times\Gamma_{+}$.
Now, to proceed further, we calculate the lower order terms of the asymptotic expansion of $a_R$.
Here we take the opportunity to define the notation $R_J$ \emph{denotes the terms on $\R_{+}\times\Gamma_{\pm}$ which are determined from the fact $(a_R)_J = (\tilde{a}_R)_J$ at $\R_{+}\times\Gamma_{-}$}. For instance, $\rho$, $\kappa$, $(a_T)_0$ and their tangential derivatives are $R_0$ terms (see Proposition \ref{Prop_recovery_0}). Also since $\rho,\kappa$ are known on $\mathcal{O}_{-}$ (see Theorem \ref{Main_th_1}), we have $\p_{x_3}(a_R)_0$,$\p_{x_3}(a_I)_0$ are also $R_0$ terms. Similarly, we will see that $\p_{x_3}^J \rho_+, \p_{x_3}^J c_+, \p_{x_3}^J \Phi_+$ are $R_J$ terms. Write $F\sim R_J$ to denote $F$ is a $R_J$ term.

We state the following result regarding the derivatives of the parameters, which is similar to \cite[Lemma 2.2]{BHKU_1} and will be exactly same if $A=0$ at $\Gamma$.
\begin{Proposition}\label{Prop_a-1}
	If $(a_R)_{-1} = (\tilde{a}_R)_{-1}$ on $\R_{+}\times\Gamma_{-}$, then
\[ 	\p_{x_3}\rho = \p_{x_3}\tilde{\rho}, \quad
	\p_{x_3}c = \p_{x_3}\tilde{c}\quad \mbox{and}\quad
 \dPhi = \nabla_x \tilde \Phi
	  \quad \mbox{on}\quad\Gamma_{+}.
\]
\end{Proposition}
\begin{proof}
Let $(a_R)_{-1} = (\tilde{a}_R)_{-1}$ on $\R_+\times\Gamma_-$. 
Next, note that $M$, $(h_I)_{-1}$ and $(h_R)_{-1}$ are $R_0$ termssince they do not contain any normal derivatives of $c_+,\rho_+$. 

By setting $J=-1$ in \eqref{e: Jth transmission cond} we obtain
\begin{multline*}
(\alpha_R)_{-1}
= -M_{11}^{-1} \langle \nu, (h_T)_{-1} \rangle 
\\
+ M_{12}^{-1}(\kappa_+ \nabla_x \cdot ((\alpha_T)_0 N_T)
-\rho_+ (\alpha_T)_0 \langle N_T, \dPhi \rangle + R_0
\end{multline*}
where $M^{-1}_{ij}$ is the $ij$'th entry of $M^{-1}$, so we need to carefully compute $ (h_T)_{-1}.$

We have using \cite[next to equation (45)]{RachDensity}
\begin{multline*}
(h_{T})_{-1}
= 
\frac{-i (\alpha_T)_0}{c^2|\xi_T|}
[ 4(N_T \ominus c_+^2 \nabla_x \log \sqrt \rho_+)N_T
+ 2(N_T \ominus c_+^2 \nabla_x \log (\alpha_T)_0) N_T + c_+^2(\nabla_x N_T)N_T]
\\
=\frac{-i}{c_+^2 |\xi_T|}
\Bigg[
-2(\alpha_T)_0 c_+^2 
(\nabla_x \log \sqrt \rho_+ -
\langle N_T, \nabla_x \log \sqrt \rho_+ \rangle) N_T
\\
-c_+^2(\nabla_x (\alpha_T)_0
- \langle N_T, \nabla_x (\alpha_T)_0 \rangle N_T)
\\
+(\alpha_T)_0 c_+^2 (\nabla_x N_T) N_T
\Bigg],
\end{multline*}
where $(v \ominus w)u
= \frac 1 2 v (w \cdot u) - \frac 1 2 w (v \cdot u)$, and since $h_T \cdot N_T = 0$, the $\ominus$ operation in the above formula projects a vector to the image of $\pi_{N_T}^\perp$. We may then compute
\begin{multline*} \label{a: h_T dot nu}
\langle (h_T)_{-1}, \nu \rangle
= \frac{2i}{|\xi_T|}
(1-\hat \xi_{T,3}^2)(\alpha_T)_0 \p_{x_3} \log \sqrt \rho
+
\frac{i}{|\xi_T|}(1-\hat \xi_{T,3}^2)\p_{x_3}(\alpha_T)_0
\\
-\frac{i(\alpha_0)_T}{|\xi_T|}
\cdot
\frac{\p_{x_3}\xi_{T,3} \xi_{T,3}}{|\xi_T|^2}
(1 - \hat \xi_{3,T}^2)
+R_0.
\end{multline*} \label{a: j+2 derivs of varphi}
We will also need \cite{RachBoundary}
\begin{equation}
\p_\nu^{j+2}
\varphi |_\Gamma
= -(\p_\nu^{j+1} \log c)
\frac{|\xi|^2}{\xi_3} + R_j
\end{equation}
and
\begin{equation}\label{a: derivative of alpha_0}
\p_\nu (\alpha_T)_0
= -\Bigg[ \frac 1 2 \p_\nu \log c_+ \left(1-\frac{|\eta'|^2}{\xi_{T,3}^2}\right)
+ \p_\nu \log \sqrt \rho \Bigg](\alpha_T)_0 + R_0
\end{equation}
since $(a_T)_0$ is independent of $\Phi$.

 The next term we compute is
\begin{multline*}
\kappa_+ \nabla_x \cdot (\alpha_T)_0 N_T
=
\kappa_+ \p_{x_3} (\alpha_T)_0\hat \xi_{T,3} +
\kappa_+ (\alpha_T)_0 \p_{x_3} \hat \xi_{T,3} +R_0
\\
=
\kappa_+ \p_{x_3} (\alpha_T)_0\hat \xi_{T,3} -
\kappa_+ (\alpha_T)_0 \frac{\p_{x_3} \log c_+}{\hat \xi_{T,3}} (1 - \hat \xi_{T,3}^2)+R_0
\end{multline*}
We then have $M_{11}^{-1} = D^{-1}(-i\kappa_+|\xi_T|)$
and $M_{12}^{-1} = -D^{-1} \hat \xi_{T,3}$ where $D = det(M)$ is an $R_0$ term.
Thus, we finally obtain
\begin{multline*}
D(\alpha_R)_{-1}
=
-2\kappa_+
(1-\hat \xi^2_{T,3}) (\alpha_T)_0 \p_{x_3}\log \sqrt \rho
\\
-\kappa_+(1-\hat \xi^2_{T,3}) \p_{x_3}(\alpha_T)_0
\\
-\kappa_+(\alpha_T)_0 (1-\hat \xi^2_{T,3})\p_{x_3}\log c \\
-\kappa_+ \p_{x_3}(\alpha_T)_0 \hat \xi^2_{T,3}
+ \kappa_+ (\alpha_T)_0
\p_{x_3}\log c_+ (1-\hat \xi^2_{T,3})
 -\rho_+ (\alpha_T)_0 \langle N_T, \dPhi \rangle \hat \xi_{T,3} + R_0
\\
= 
 -\Bigg[ \frac 1 2 \p_\nu \log c_+ \left(1-\frac{3|\eta'|^2}{\xi_{T,3}^2}\right)
+ \p_\nu \log \sqrt \rho \Bigg](\alpha_T)_0 
\\
-\rho_+ (\alpha_T)_0 \langle N_T, \dPhi \rangle \hat \xi_{T,3}+ R_0.
\end{multline*}

Therefore, $(a_R)_{-1} = (\tilde{a}_R)_{-1}$ on $\R_{+}\times\Gamma_{-}$ implies
\begin{multline}\label{e: a_r = tilde a_r order -1}
-\Bigg[ \frac 1 2 \p_\nu \log c_+ \left(1-\frac{3|\eta'|^2}{\xi_{T,3}^2}\right)
+ \p_\nu \log \sqrt \rho \Bigg] 
\\
-\rho_+  \langle N_T, \dPhi \rangle \hat \xi_{T,3}
\\
= 
-\Bigg[ \frac 1 2 \p_\nu \log \tilde c_+ \left(1-\frac{3|\eta'|^2}{\xi_{T,3}^2}\right)
+ \p_\nu \log \sqrt {\tilde \rho} \Bigg] 
-\rho_+ \langle N_T, \nabla_x \tilde \Phi  \rangle \hat \xi_{T,3}
\end{multline}
This holds for any covector $(\eta',\tau)$ in the hyperbolic/hyperbolic set. Pick $(\eta'_{(j)}, \tau_{(j)})$, $j = 1, 2$
such that $|\eta'_{(1)}| = |\eta'_{(2)}|$
and $\tau_{(1)} = \tau_{(2)}$. We apply \eqref{e: a_r = tilde a_r order -1} for both covectors and subtract the two equations to obtain (with the terms involving $c$ and $\rho$ cancelling out)
\begin{equation}
(\eta_{1,(1)} - \eta_{1,(2)})(\p_{x_1}\Phi - \p_{x_1}\tilde \Phi)
+ (\eta_{2,(1)} - \eta_{2,(2)})(\p_{x_2}\Phi - \p_{x_2}\tilde \Phi)) = 0.
\end{equation}
Pick $\eta_{1,{1}} = - \eta_{1,(2)}$
and $\eta_{2,{1}} =  \eta_{2,(2)}$ in the above equation to conclude
$\p_{x_1} \Phi - \p_{x_1}\tilde \Phi$ at $\Gamma$. Similarly, we can obtain $\p_{x_2} \Phi - \p_{x_2}\tilde \Phi$.
Next, choose $\eta' = 0$ in \eqref{e: a_r = tilde a_r order -1} but $\tau_{(j)} \neq 0$ and subtract the two equations to get
\begin{equation}
(\xi_{T,3, (1)} - \xi_{T,3,(2)})(\p_{x_3}\Phi - \p_{x_3}\tilde \Phi) = 0.
\end{equation}
Thus, we conclude that $\dPhi|_{\Gamma}$ can be recovered from the symbol of $M_R$ restricted to the hyperbolic set. Note that $\Phi$ and $\p_\nu \Phi$ is continuous across the interface due to the elliptic equation it solves, but the above argument shows that this is not needed in order to recover $\nabla_x \Phi$ across the interface.

Next, to recover $\p_\nu \log c_+$, we consider any $(\eta'_{(j)}, \tau_{(j)})$ such that $|\eta'_{(1)}| \neq |\eta'_{(2)}|$ in \eqref{e: a_r = tilde a_r order -1}, and substracting the two equations from each other gives
\begin{equation} \left(\frac{3|\eta'_{(1)}|^2}{\xi_{T,3,(1)}^2} - \frac{3|\eta'_{(2)}|^2}{\xi_{T,3,(2)}^2}
\right)(\p_\nu \log c_+ - \p_\nu \log \tilde c_+) = 0 
\end{equation}
so $\p_\nu \log c_+ = \p_\nu \log \tilde c_+$ on $\Gamma$. Since $\p_\nu \log c_+$ and $\nabla_x \Phi$ have been recovered in \eqref{e: a_r = tilde a_r order -1}, then we trivially conclude $\p_\nu \log \rho_+ = \p_\nu \log \tilde \rho_+$.
\end{proof}

Now, we have the parameters $c$, $\rho, \Phi$ and their first order normal derivatives determined at $\Gamma$. 
We proceed further in a similar fashion, i.e., considering the term $(a_R)_J$ for $J\leq -2$ and obtain a relation between $\p_{x_3}^{|J|}(a_T)_0$ and $\p_{x_3}^{|J|}(\tilde{a}_T)_0$ from the transport equation, at $\R_+\times\Gamma_+$.
From there, using asymptotic analysis we determine higher order derivatives of $c$, $\rho$ at $\Gamma$. We have the following result.
\begin{Proposition}\label{Prop_a-2}
	Suppose $c = \tilde c$ in $\mathcal O$, $\rho = \tilde \rho$ and $\Phi = \tilde \Phi$ in $\mathcal O_-$, and suppose $a_I$ is known on $\R_+\times\overline{\mathcal{O}}$.
	If $(a_R)_{J} = (\tilde a_R)_J$ on $\R_+\times\Gamma_{-}$ for $J=0,-1,-2,-3, \dots$; then one can uniquely determine $\p_{x_3}^{|J|}\rho_+$, $\p_{x_3}^{|J|} \Phi_+$, on $\Gamma_{+}$.
\end{Proposition}
\begin{proof}
	Let us assume $(a_R)_J = (\tilde{a}_R)_J$ at $\R_+\times\Gamma_-$ for $J=0,-1,-2$.
	From the transmission condition \eqref{e: Jth transmission cond} and the proof of the previous proposition implies that $\text{div} (a_T)_{-1} |_\Gamma \nu $ is uniquely determined from $(a_R)_{-2}$ since all other terms are $R_1$ terms, having already determined $\dPhi.$ That is, \begin{equation}\label{e: order -2 a_T equation}
\p_{x_3} (a_T)_{-1}|_{\Gamma} = \p_{x_3} (a_T)_{-1}|_{\Gamma}.
    \end{equation}

    The key point is that $\p_\nu^2 \Phi$ can be determined from the elliptic PDE that $\Phi$ solves, so that all $\Phi$ terms in \eqref{e: order -2 a_T equation} can be cancelled and do not appear. We will now provide the proof of this. First, we have $\Delta \Phi = k \rho$ for a uniform constant $k$. This equation using boundary normal coordinates near $\Gamma$ is
    \[
\p_\nu^2 \Phi|_{\Gamma} = k \rho + (\text{ terms involving at most 1 normal derivative of $\Phi$}) = R_1.
    \]
    The right side of the equation also involve terms containing the curvature of $\Gamma$ and its derivatives but those have already been determined from the recovery of $c$, which in particular, recovers $\Gamma$ uniquely.
    By applying $\p_\nu^{|J|}$ to the above equation, we obtain more generally that
    \begin{equation}\label{e: p^J Phi is R_J-1}
\p_\nu^{|J|} \Phi|_{\Gamma}
= R_{|J|-1}.
    \end{equation}

We transmission conditions \eqref{e: Jth transmission cond} show that $(a_T)_{J-1}$ is determined by $\p_{x_3} (a_T)_{J}$ and terms belonging to $R_{|J|}$. Now, $\p_{x_3} (a_T)_{J}$ will have at most $|J|$ normal derivatives of $\rho_+$, $c_+$, and $\Phi_+$. The $\Phi$ calculation above shows that in fact, $\p_{x_3} (a_T)_{J}$ depends on at most $|J|-1$ normal derivatives of $\Phi$. Since we proceed inductively, having already recovered $\p_{x_3}^{|J|-1} \rho_+$, $\p_{x_3}^{|J|-1} c_+$, $\p_{x_3}^{|J|-1} \Phi_+$ on $\Gamma$ from the previous step, we just need to compute the dependence of $\p_{x_3} (a_T)_{J}$ on $\p^{|J|}\rho_+$ and $\p^{|J|}c_+$ since all other terms will be in $R_{|J|}$.

Recall the transport equation that allows us to compute $(\dot{\alpha}_T)_J$:
\[
2i\kappa|\xi_T| (\dot{\alpha})_J
=2i\kappa |\xi_T|\hat \xi_{T,3}
(\p_{x_3} \alpha_T)_J + R_{|J|}
\]
\begin{multline}
=-\Bigg[
(\alpha_T)_J N_T \cdot B_{1,1}(N_T)
+
N\cdot \mL(x,D)  \big( (\alpha_T)_{J+1}N_T + (h_T)_{J+1}\big)
\\
+N_T \cdot B_{1,1}(h_T)_{J}
+ N_T \cdot P_{sc}(\varphi_T)((\alpha_T)_{J+1}N_T + (h_T)_{J+1})
\Bigg]
\\
= -\Bigg[
N_T \cdot B_{1,1}(h_T)_{J}
+ N_T \cdot \mL(x,D)
\big( (\alpha_T)_{J+1}N_T + (h_T)_{J+1}\big) \Bigg]
+ R_{|J|}
\end{multline}
where we have written $B_p$ in \eqref{e: B_p term} as $B_p = B_{1,1} + P_{sc}(\varphi_T)$, such that

\begin{align*}
B_{1,1} M
&= i\p_{\tau,\xi}p(t,x,\p_{t,x}\varphi_T) \cdot \p_{t,x} M
+i(ip_1)(t,x,\p_{t,x}\varphi_T)M
\\
P_{sc}(\varphi_T) M &= \frac{i}{2}\sum_{|\alpha|=2}\sum_{l=1}^3
(\p_{\tau,\xi}^\alpha p^{il})(t,x,\p_{t,x}\phi)
\cdot (\p^\alpha_{t,x} \phi) M^l,
\end{align*}

After a computation, we obtain
\begin{multline}
\p_{x_3}(\alpha_T)_J
= 
-\frac{1}{2\rho_+ c_+^2 \xi_{T,3}} \Bigg[
\kappa_+ \p_{x_3}(h_T)_{J}\cdot N_T |\xi_{T,3}|
+i \kappa_+\p_{x_3}^2 (h_T)_{J+1} \cdot N_T
\\
+\p_{x_3}^2(\alpha_T)_{J+1} \kappa_+ \hat \xi^2_{T,3} - (\alpha_T)_{J+1} \kappa_+\p_{x_3}^2 \log c_+  \frac{|\eta'|^2}{|\xi_T|^2} 
+ \langle F_{\Phi, J+1}, N_T \rangle \Bigg]+ R_{|J|},
\end{multline}
where
\begin{multline}
\langle F_{\Phi,J}, N_T \rangle
= \Big(-\nabla_x \langle \dPhi, \rho u_T \rangle + \nabla_x \cdot (\rho_+ u_T) \dPhi \Big)_J
\\
=
-\rho_+ \nabla_x^2 \Phi( (\alpha_T)_J N_T + (h_T)_{J},N_T)
-\rho_+ (\alpha_T)_J N_T \cdot \nabla_x N_T (\dPhi)
\\
-\rho_+\nabla_x h_{J}(\nabla_x \Phi, N_T)
+ \langle N_T, \dPhi \rangle (\rho_+ (\alpha_T)_{J} \nabla_x \cdot N_T + \rho_+ \nabla_x \cdot (h_T)_J) + R_{|J|-1},
\end{multline}
using the notation $A(V, W) = W\cdot (AV)$ for a matrix $A$ and vectors $V,W$.
In particular
\[
\langle F_{\Phi,0}, N_T \rangle = - \p_{x_3}^2 \Phi \rho_+ (\alpha_T)_0 \hat \xi_{T,3}^2 +R_1 = R_1
\]
since we showed $\p_{x_3}^2 \Phi$ is an $R_1$ term. Hence, using $J=-1$ in the above formulas, since $\p_\nu^2 \Phi$ has already been recovered, we write down each term that will contain either $\p_\nu^2 \log c$ or $\p_\nu^2 \log \rho$:
\begin{multline}
R_1 |\xi_R| (\alpha_R)_{-2}
= \kappa_+ \nabla_x \cdot (a_T)_{-1} + R_1
= 
\kappa_+ \p_{x_3}(\alpha_T)_{-1} \hat \xi_{T,3} + \kappa_+\p_{x_3} (h_T)_{-1} \cdot \nu + R_1
\\
=
\frac{\kappa_+}{|\xi_{T}|}
\Big[i \p_{x_3}^2 (\alpha_T)_0 \hat \xi_{T,3}^2 
- i(\alpha_T)_0 \p^2_{x_3} \log c_+ \frac{|\eta'|^2}{\xi_{T,3}^2} \hat \xi_{T,3}^2\Big]
\\
+\frac{i(\alpha_T)_0}{|\xi_T|}
\Big[\p_{x_3}^2\kappa_+(1-\hat \xi_{T,3}^2)
+\kappa_+ \p^2_{x_3}(\alpha_T)_0 (1-\hat \xi_{T,3}^2)
+\kappa_+ \p_{x_3}^2|\xi_T|
\Big]
+ R_1,
\end{multline}
which agrees with the computation for $\p_\nu (\alpha)_{-1}$ and $\p_\nu (h)_{-1}$ in \cite{RachBoundary}.

This simplifies to
\begin{multline}
\frac{i\kappa_+ (\alpha_T)_0}{|\xi_T|}\left[\frac{1}{2}\p^2_{x_3} \log c_+ \left(1-\frac{2|\eta'|^2}{|\xi_T|^2}\right) + \p_{x_3}^2 \log \sqrt \rho_+ 
\right]
\\
+
\frac{i\kappa_+}{|\xi_T|}\left( \p^2_{x_3} \log \rho_+ + 
\p_{x_3}^2 \log c_+ \right) (1-\hat \xi_{T,3}^2) + R_1.
\end{multline}
Hence, we see that due to the $\p_{x_3}(h_T)_{-1}$ term, we cannot recover both $\p_{x_3}^2 \log c_+$ and $\p_{x_3}^2 \log \rho_+$ simultaneously from $(\alpha_R)_{-2}.$ This is why we need to have already recovered $c_+$ before we recover the derivatives of $\rho_+$, as we do in the main theorem of the paper using local travel time tomography. Hence, with $c_+$ recovered, we may recover $\p_\nu^2 \log \rho_+$ using the same argument for the first normal derivative.

Continuing inductively, we may show that $(\alpha_R)_{J}$ is an $R_{|J|}$ term for $J = -1, -2, \dots$ and then recover $\p_\nu^{|J|} \log \rho_+$ from $(\alpha_R)_{J}.$ by the analogous argument by using \eqref{e: p^J Phi is R_J-1} to treat $\p^{|J|}_\nu \Phi$ terms as lower order. 
\end{proof}

Now let us consider the original operator with $B(x,D)$. Since $B(x,D)$ is zeroth order, the same proof as above goes through but it requires a careful justification that we now provide.
\begin{enumerate}
\item[$\bullet$]
The natural transmission conditions \eqref{e: trans conditions} are unchanged since $B$ is zeroth order. Hence, equations $\eqref{transmission_cond_J}$ and $\eqref{e: Jth transmission cond}$ remain unchanged.
\item[$\bullet$] We saw in section \ref{Sec_geometric_cond} that $(a_\bullet)_0$ is independent of $B$ so proposition \ref{Prop_recovery_0} continues to hold with the same proof.

\item[$\bullet$] By \eqref{transmission_cond_J} and \eqref{e: bdy conditions of a_I}, $(a_R)_{-1}|_\Gamma$ depends only on $\p_\nu (a_\bullet)_0|_\Gamma$ for $\bullet = I,R,T$ and not the lower order amplitudes. Since these quantities are independent of $B$, proposition \ref{Prop_a-1} holds with the same proof, so that $\rho, \p_\nu \rho,  \dPhi, c, \p_\nu c$ are recovered on the other side of the interface with no change in the proof. 
\item[$\bullet$]
Similary, $(a_R)_{-2}|_\Gamma$ depends only on $\p_\nu (a_\bullet)_{-1}|_\Gamma$ and $(a_\bullet)_0$ for $\bullet = I,R,T$ and not the lower order amplitudes. Looking at \eqref{transport_eq_Interior}, $\p_\nu (a_{\bullet})_{-1}$ depends on the on the zeroth order term of the symbol of $\rho B \rho$ which is $\rho^2 b_0$ and not the lower order terms in the expansion. Since $\rho$ has already been determined at the interface, the formula for $(a_R)_{-2}$ in the above proof remains the same modulo $R_0$ terms. Hence, the proof of proposition \ref{Prop_a-2} remains valid to recover $\p_\nu^2 \rho|_{\Gamma^+}$. 

\item[$\bullet$]Before generalizing, let us consider $(a_R)_{-3}$.
Again,  $(a_R)_{-3}|_\Gamma$ depends only on $\p_\nu (a_\bullet)_{-2}|_\Gamma$ for $\bullet = I,R,T$ and not the lower order amplitudes. Looking at \eqref{transport_eq_Interior}, $\p_\nu (a_{\bullet})_{-2}$ depends on the on the symbol of order $-1$ in the asymptotic expansion of the symbol of $\rho B \rho$. By lemma \ref{l: asymp expansion of P applied to lagrangian distribution}, this will involve $\rho$ and $\nabla_x \rho$ and not any higher order derivatives of $\rho$, and these have already been determined from the previous steps. Thus, the proof of proposition \ref{Prop_a-2} for $J=-3$ remains valid to recover $\p_\nu^3 \rho|_{\Gamma^+}$. 

\item[$\bullet$]
Proceeding inductively from the previous steps, 
 $(a_R)_{-J}|_\Gamma$ depends only on $\p_\nu (a_\bullet)_{-J+1}|_\Gamma$ for $\bullet = I,R,T$ and not the lower order amplitudes. Looking at \eqref{transport_eq_Interior}, $\p_\nu (a_{\bullet})_{-J+1}$ depends on the symbol of order $-J+2$ in the asymptotic expansion of the symbol of $\rho B \rho$. By lemma \ref{l: asymp expansion of P applied to lagrangian distribution}, this will involve $\rho, \nabla_x, \dots, \nabla_x^{J-2} \rho$ and not any higher order derivatives of $\rho$, and these have already been determined from the previous steps. Thus, the proof of proposition \ref{Prop_a-2} for each $J$ remains valid to recover $\p_\nu^J \rho|_{\Gamma^+}$.

\end{enumerate}

Here, we end our discussion on recovery across interfaces with a final remark.
\begin{Remark}
	 In order to complete the proof of the Theorem \ref{Main_th_1} we observe that $(a_T)_J$ can be recovered at $\R_+\times\Gamma_+$ from the knowledge of the parameters and their derivatives at $\Gamma_+$ along with the transmission condition \eqref{transmission_cond_J}, when $a_I$ and $a_R$ are known in $\R_+\times\overline{\mathcal{O}}_-$.
	 Therefore, one can determine the transmission operator $M_T$ at $\R_+\times\Gamma$.
\end{Remark}

\section{Proof of lemma \ref{l: calculating G}} \label{a: proving formula for A(x)}
In this appendix, we do the computations to prove lemma \ref{l: calculating G} that we break into several parts.
\subsection*{Computing \texorpdfstring{$h_{-1}$}{Lg}}
We have the operator lower order principal symbol
\[
p_1(t,x,\p_x \phi)
=-i\nabla_x \lambda \otimes \nabla_x \phi + \rho i \p_x \phi \otimes \nabla_x \Phi - \rho i \nabla_x \Phi \otimes \nabla_x \phi 
\]
so that we compute a constituent of $h_{-1}$
\[
\pi_p^\perp p_1(t,x,\p_{t,x}\phi)(\alpha)_0 N
= -i |\xi| \nabla_x \lambda (\alpha)_0 - i \rho |\xi| \nabla_x \Phi (\alpha)_0,
\]
so the last term is the only term containing $\Phi$ when computing $h_-1$.

Since $h_{-1} \cdot N =0$, we compute explicitly from \cite[below equation (45)]{RachDensity}
\begin{multline}
h_{-1}= \frac{-1}{\rho c^2 |\xi|^2}B_p( (\alpha)_0 N) \\
= \frac{-1}{\rho c^2 |\xi|^2} \cdot i \rho |\xi|\Bigg(
2c \p_t((\alpha)_0 N)+ c^2\left[ [\nabla_x \otimes (\alpha_0 N)]N + [\nabla_x \cdot (\alpha_0 N)] N   \right]\\
+ \left[2c^2 \nabla_x \log \sqrt{\rho c^2}
+ c^2 \frac{[\nabla_x \otimes \xi]}{|\xi|} N + \nabla_x \Phi \right](\alpha)_0
\Bigg)
\end{multline}

\subsection*{Computing \texorpdfstring{$N \cdot C_p (\alpha)_0 N$}{Lg}}

First we compute the part of $p_1(t,x,D_x)(\alpha_0 N)$ that depends on $\Phi$:
\begin{multline}
p_{1, \Phi}(t,x,D_x) (\alpha_0 N) = \nabla (\rho \alpha N \cdot \nabla \Phi) - \nabla \cdot (\rho \alpha_0 N) \nabla_x \Phi 
\\
=  N \frac{H}{\sqrt \rho} (\nabla_x \rho \cdot \nabla \Phi)
+ N \rho \left(  \frac{H}{\sqrt \rho} -  \frac{H}{\sqrt \rho} \nabla_x \log \sqrt \rho \right) \cdot \nabla \Phi
\\
+ \rho  \frac{H}{\sqrt \rho} (\nabla_x \otimes N)\nabla_x \Phi
-  \frac{H}{\sqrt \rho} (\nabla_x \rho \cdot N) \nabla_x \Phi 
- \rho \left( \frac{H}{\sqrt \rho}-  \frac{H}{\sqrt \rho}\nabla_x \log \sqrt \rho \right) \cdot N \nabla_x \Phi
\\
-\sqrt \rho H (\nabla_x \cdot N)\nabla_x \Phi
+ \sqrt \rho H (\nabla_x^2\Phi N)
\\
= \sqrt \rho H \Bigg(
2N (\nabla_x \log \sqrt \rho \cdot \nabla_x \Phi )
+ N \left(  \frac{\nabla_x H}{H} \cdot \nabla_x \Phi - \nabla_x \log \sqrt \rho \cdot \nabla_x \Phi\right) 
\\
+ (\nabla_x \otimes N ) \nabla_x \Phi
\\
-2\langle \nabla_x \log\sqrt \rho \cdot N\rangle \dPhi
- \left( \frac{\nabla_x H}{H} \cdot N - \nabla_x \log \sqrt \rho \cdot N \right) \dPhi
\\
- (\nabla_x \cdot N) \dPhi + (\nabla_x^2 \Phi)N \Bigg).
\end{multline}

Thus, 
\begin{multline}
N\cdot p_{1,\Phi}(\alpha_0 N) 
= \sqrt \rho H \Bigg(
\nabla_x \log \sqrt \rho \cdot \dPhi +  \frac{\nabla_x H}{H} \cdot \dPhi 
+ \dPhi \cdot (\nabla_x \otimes N)^t N
\\
- \langle \nabla_x \log \sqrt \rho, N \rangle \langle \dPhi, N \rangle- \langle  \frac{\nabla_x H}{H}, N \rangle \langle \dPhi, N \rangle \\
- (\nabla_x \cdot N) \langle \dPhi , N\rangle
+ \nabla_x^2 \Phi(N, N) \Bigg)
\end{multline}

Using \cite[equation (43)]{RachDensity}
we obtain (with terms containing $\Phi$ highlighted)
\begin{multline}
N \cdot C_p (\alpha_0 N) \\
=-H\sqrt \rho \Bigg( c^2 [N \nabla_x^2\log \sqrt \rho N]
+ c^2[(\log \sqrt{\rho})']^2
\\
+ (\log \sqrt{\rho})' \left[2c^2 (\log \sqrt{c^2})'
-c^2 (\nabla_x \cdot N)\right]
\\
+ \nabla_x \log \sqrt{\rho} \cdot c^2 (\nabla_x \otimes N)^t N
\\
\textcolor{red}{+\nabla_x \log \sqrt \rho \cdot \dPhi +  \frac{\nabla_x H}{H} \cdot \dPhi 
+ \dPhi \cdot (\nabla_x \otimes N)^t N}
\\
\textcolor{red}{- \langle \nabla_x \log \sqrt \rho, N \rangle \langle \dPhi, N \rangle- \langle  \frac{\nabla_x H}{H}, N \rangle \langle \dPhi, N \rangle} \\
\textcolor{red}{- (\nabla_x \cdot N) \langle \dPhi , N\rangle
+ \nabla_x^2 \Phi(N, N) }
+ \text{ terms that do not depend on $\rho$ or $\Phi$} \Bigg)
\end{multline}

\subsection*{Computing \texorpdfstring{$N\cdot B_p h_{-1}$}{Lg}}

First, since $N \cdot h_{-1}=0$, then
\begin{multline}
N \cdot P_{1,1}h_{-1} = 
-i\lambda N \cdot [ \nabla_x \phi \otimes \nabla_x + \nabla_x \otimes \nabla_x \phi] h_{-1}
=-i \lambda \nabla_x \phi  \nabla_x \cdot h_{-1}
\\
= \sqrt \rho H \Bigg[\textcolor{red}{ c^2 \frac{|\xi|}{H}\nabla_x \frac{H}{c^2|\xi|} \cdot \dPhi- \nabla_x \log \sqrt \rho \cdot \dPhi + \Delta_x \Phi} \Bigg]\\
+  \text{ (terms that do not depend on $\Phi$)}
\end{multline}
Also
\begin{multline} \label{e: P_{0,1}h_-1}
N \cdot P_{0,1}(t,x,\p_{t,x}\phi)h_{-1}
= N \cdot i(-\nabla_x \lambda \otimes \nabla_x \phi + \rho \nabla_x \phi \otimes \dPhi
- \rho \dPhi \otimes \nabla_x \phi) h_{-1}
\\
= \sqrt\rho H( \textcolor{red}{-c^{-2}|\dPhi|^2} + \text{ (terms that do not depend on $\Phi$)}
\end{multline}
Finally, we have the additional term
\[
N \cdot \rho b_0(t,x, \p_x \phi) \rho \alpha_0 N = \sqrt \rho H  (-\rho)
\]
coming from the self-gravitation.

We combine the above pieces to obtain the desired result
\begin{align}
g(\alpha)_{-1} &=\int gG + C \nonumber \\
&= \int \frac{\sqrt \rho}{H}\cdot \frac{-1}{2i\rho c^2|\xi|} \cdot(-H\sqrt \rho)[N^t \cdot A(x) \cdot N] + C
\end{align}
where the rank two tensor $A(x)$ depends only on $x$ and is given by
\begin{multline}
A(x) = c^2 \nabla_x^2 \log \sqrt \rho 
+ 2\nabla_x \log \sqrt \rho \otimes \nabla_x c^2
\\
- I \Bigg[ c^2 \Delta(\log \sqrt \rho)] 
-c^2|\nabla_x \log \sqrt \rho |^2
+ \nabla_x \log \sqrt \rho \cdot \nabla_x c^2 \Bigg]
\\
\textcolor{red}{- \nabla_x \log \sqrt \rho \otimes \dPhi-   \frac{\nabla_x H}{H}  \otimes \dPhi + (\nabla_x \otimes N)\dPhi \otimes N}\\
\textcolor{red}{- (\nabla_x \cdot N)  \dPhi \otimes N
+ \nabla_x^2 \Phi}
\\
\textcolor{red}{+I\Bigg[c^2 \frac{|\xi|}{H}\nabla_x \frac{H}{c^2|\xi|} \cdot \dPhi- \nabla_x \log \sqrt \rho \cdot \dPhi + \Delta_x \Phi - c^{-2}|\dPhi|^2 - \rho} \Bigg]
\\
+(\text{ terms independent of $\rho$ and $\Phi$}).
\end{multline}

\bibliographystyle{amsplain}
\bibliography{ElasticFluidCase.bib}

\end{document}